%% file: lognag_array_easy.tex
\RequirePackage[l2tabu,orthodox]{nag}	

\documentclass[a4paper,12pt,english]{article}

	\usepackage{lmodern}				

	\usepackage[T1]{fontenc}			

	\usepackage[utf8x]{inputenc}		

\input{packmin}

\input{commath}

	\usepackage{pstricks}
	\usepackage{ifmtarg}
	\usepackage{hyperref}
	\usepackage{cleveref}
	\usepackage{authblk}
	\usepackage{tikz}
	\usetikzlibrary{patterns}

\makeatletter
\newcommand*{\rva}[3][]{%
	\@ifmtarg{#1}{
		#2_{#3}
	}{%
		#2_{#3,#1}
	}%
}

\let\orgdescriptionlabel\descriptionlabel
\renewcommand*{\descriptionlabel}[1]{%
  \let\orglabel\label
  \let\label\@gobble
  \phantomsection
  \edef\@currentlabel{#1}%
  \let\label\orglabel
  \orgdescriptionlabel{#1}%
}
\makeatother

\theoremstyle{plain}
\newtheorem{thm}{Theorem}

\newtheorem{lem}[thm]{Lemma}

\newcounter{assumption}
\theoremstyle{definition}
\newtheorem{rem}[thm]{Remark}

\newcommand*{\Yinf}{Y^<}
\newcommand*{\Tinf}{T^<}

\makeatletter
\tikzset{
        hatch distance/.store in=\hatchdistance,
        hatch distance=5pt,
        hatch thickness/.store in=\hatchthickness,
        hatch thickness=5pt
        }
\pgfdeclarepatternformonly[\hatchdistance,\hatchthickness]{north east hatch}
    {\pgfqpoint{-1pt}{-1pt}}
    {\pgfqpoint{\hatchdistance}{\hatchdistance}}
    {\pgfpoint{\hatchdistance-1pt}{\hatchdistance-1pt}}%
    {
        \pgfsetcolor{\tikz@pattern@color}
        \pgfsetlinewidth{\hatchthickness}
        \pgfpathmoveto{\pgfqpoint{0pt}{0pt}}
        \pgfpathlineto{\pgfqpoint{\hatchdistance}{\hatchdistance}}
        \pgfusepath{stroke}
    }
\pgfdeclarepatternformonly[\hatchdistance,\hatchthickness]{north west hatch}
    {\pgfqpoint{4pt}{4pt}}
    {\pgfqpoint{\hatchdistance}{\hatchdistance}}
    {\pgfpoint{\hatchdistance-1pt}{\hatchdistance-1pt}}%
    {
        \pgfsetcolor{\tikz@pattern@color}
        \pgfsetlinewidth{\hatchthickness}
        \pgfpathmoveto{\pgfqpoint{0pt}{0pt}}
        \pgfpathlineto{\pgfqpoint{\hatchdistance}{\hatchdistance}}
        \pgfusepath{stroke}
    }
\makeatother

	\usepackage{hyperref}				
	\hypersetup{colorlinks=true,linkcolor=black,urlcolor=black,citecolor=black}
	\usepackage{cleveref}				

\usepackage{geometry}				
\begin{document}
\title{Large deviation results for triangular arrays of semiexponential random variables}
\author[1]{Thierry Klein}
\author[2]{Agn\`es Lagnoux}
\author[3]{Pierre Petit}

\affil[1]{Institut de Math\'ematiques de Toulouse; UMR5219. Universit\'e de Toulouse; ENAC - Ecole Nationale de l'Aviation Civile , Universit\'e de Toulouse, France}
\affil[2]{Institut de Math\'ematiques de Toulouse; UMR5219. Universit\'e de Toulouse; CNRS. UT2J, F-31058 Toulouse, France.}
\affil[3]{Institut de Math\'ematiques de Toulouse; UMR5219. Universit\'e de Toulouse; CNRS. UT3, F-31062 Toulouse, France.}

\maketitle

\begin{abstract}
Asymptotics deviation probabilities of the sum $S_n=X_1+\dots+X_n$ of independent and identically distributed real-valued random variables have been extensively investigated, in particular when $X_1$ is not exponentially integrable. For instance, A.V.\ Nagaev formulated exact asymptotics results for $\Prob(S_n>x_n)$ when $X_1$ has a semiexponential distribution (\emph{see}, \cite{Nagaev69-1,Nagaev69-2}). In the same setting,  the authors of \cite{FATP2020} derived deviation results at logarithmic scale with shorter proofs relying on classical tools of large deviation theory and expliciting the rate function at the transition. In this paper, we exhibit the same asymptotic behaviour for triangular arrays of semiexponentially distributed random variables, no more supposed absolutely continuous. 
\end{abstract}





\section{Introduction}

Moderate and large deviations of the sum of independent and identically distributed (i.i.d.) real-valued random variables have been investigated since the beginning of the 20th century. 
Kinchin \cite{Kinchin29} in 1929 was the first to give a result on large deviations of the sum of i.i.d.\ Bernoulli distributed random variables. In 1933, Smirnov  \cite{Smirnov33} improved this result and in 1938 Cramér \cite{Cramer38} gave a generalization to sums of i.i.d.\ random variables satisfying the eponymous Cramér's condition which requires the Laplace transform of the common distribution of the random variables to be finite in a neighborhood of zero. Cramér's result was extended by Feller \cite{Feller43} to sequences of non identically distributed bounded random variables.
A strengthening of Feller's result was given by Petrov in \cite{Petrov54,petrov2008large} for non identically distributed random variables. 
When Cramér's condition does not hold, an early result is due to
Linnik \cite{Linnik61} in 1961 and concerns polynomial-tailed random variables. The case where the tail decreases faster than all power functions (but not enough for Cramér's condition to be satisfied) has been considered by Petrov \cite{Petrov54} and by S.V.\ Nagaev \cite{Nagaev62}.  In \cite{Nagaev69-1,Nagaev69-2}, A.V.\ Nagaev  studied the case where the commom distribution of the i.i.d.\ random variables is absolutely continuous with respect to the Lebesgue measure with density  $p(t)\sim e^{-\abs{t}^{1-\epsilon}}$ as $\abs{t}$ tends to infinity, with $\epsilon \in (0,1)$. He distinguished five exact-asymptotics results corresponding to five types of deviation speeds. In \cite{borovkov2000large, borovkov2008asymptotic}, Borovkov investigated exact asymptotics of the deviations probability for random variables with semiexponential distribution, also called Weibull-like distribution, i.e.\ with a tail writing as $e^{-t^{1-\epsilon}L(t)}$, where $\epsilon \in (0,1)$ and $L$ is a suitably slowly varying function at infinity. In \cite{FATP2020}, the authors consider the following setting. Let $\epsilon \in (0,1)$ and let $X$ be a real-valued random variable with a density $p$ with respect to the Lebesgue measure verifying:
\begin{equation}
\label{eq:behav_X}
p(x)\sim e^{-x^{1-\epsilon}},\quad \mathrm{as} \quad x\to +\infty.
\end{equation}
and
\begin{equation}
\label{hyp1}
\exists \gamma \in \intervalleof{0}{1} \quad \rho \defeq \Espe[|X|^{2+\gamma}] < \infty .
\end{equation}
For all $n \in \N^*$, let $X_1$, $X_2$, ..., $X_n$ be i.i.d.\ copies of $X$ and set $S_n=X_1+\dots+X_n$ and $P_n(x)=\Prob(S_n\geqslant x)$. According to the asymptotics of $x_n$, three logarithmic asymptotic ranges then appear. In the sequel, the notation $x_n \gg y_n$ (resp. $x_n \ll y_n$, $x_n \preccurlyeq y_n$, and $x_n =\Theta(y_n)$) means that $y_n/x_n\to 0$ (resp. $x_n/y_n\to 0$, $\limsup \abs{x_n/y_n}<\infty$, and $y_n \preccurlyeq x_n \preccurlyeq y_n$) as $n\to \infty$.
  
\begin{description}
\item[Maximal jump range] \cite[Theorem 1]{FATP2020} When $x_n \gg n^{1/(1+\epsilon)}$, 
\[
\log P_n(x_n)\sim \log\Prob(\max(X_1,\ldots,X_n)\geqslant x_n).
\]
\item[Gaussian range] \cite[Theorem 2]{FATP2020} When $x_n \ll n^{1/(1+\epsilon)}$, 
\[
\log P_n(x_n)\sim \log(1-\phi(n^{-1/2} x_n)),
\] 
$\phi$ being the cumulative distribution function of the standard Gaussian law. 
\item[Transition] \cite[Theorem 3]{FATP2020} The case $x_n = \Theta(n^{1/(1+\epsilon)})$ appears to be an interpolation between the Gaussian range and the maximal jump one.
\end{description}

In the present paper, we exhibit the same asymptotic behaviour for triangular arrays of random variables $(\rva[i]{Y}{n})_{1 \leqslant i \leqslant N_n}$ satisfying the following weaker assumption: there exists $q>0$ such that, if $y_n \to \infty$, 
\begin{align}\label{ass:queue}
\log\Prob(\rva{Y}{n} \geqslant y_n) \sim -q y_n^{1-\varepsilon} ,
\end{align}
together with similar assumptions on the moments.


\medskip

The first main contribution of this paper is the generalization of \cite{FATP2020,Nagaev69-1,Nagaev69-2} to triangular arrays. Such a setting appears naturally in some combinatorial problems, such as those presented by \cite{Janson01a}, including hashing with linear probing. Since the eighty's, laws of large numbers have been established for triangular arrays (see, e.g., \cite{gut1992complete,GUT199249,hu1989strong}). Lindeberg's condition  is standard for the central limit theorem to hold for triangular arrays (see, e.g., \cite[Theorem 27.2]{billingsley2013convergence}). Dealing with triangular arrays of light-tailed random variables, Gärtner-Ellis theorem provides moderate and large deviation results. Deviations for sums of heavy-tailed i.i.d.\ random variables are studied by several authors (e.g., \cite{borovkov2000large, borovkov2008asymptotic, FATP2020,  Linnik61, Nagaev69-1, Nagaev69-2, Nagaev79,Petrov54}) and a good survey can be found in \cite{Mikosch98}.  Here, we focus on the particular case of semiexponential tails (treated in \cite{borovkov2000large, borovkov2008asymptotic,FATP2020, Nagaev69-1,Nagaev69-2} for sums of i.i.d.\ random variables) generalizing the results to triangular arrays. See \cite{ATP2020hashing} for an application to hashing with linear probing.

Another contribution is the fact that the random variables are not supposed absolutely continuous as in \cite{FATP2020,Nagaev69-1,Nagaev69-2}. Assumption \eqref{ass:queue} is analogue to that of \cite{borovkov2000large, borovkov2008asymptotic}, but there the transition at $x_n = \Theta(n^{1/(1+\epsilon)})$ is not considered. Hence, up to our knowledge, Theorem \ref{thm:nagaev_weak_array_mob_eps_intermediate} is the first large deviation result at the transition which is explicit. 


The paper is organized as follows. In Section \ref{sec:main}, we state the main results, the proofs of which can be found in Section \ref{sec:proofs}. In Section \ref{sec:assump}, a discussion on the assumptions is proposed.
Section \ref{sec:baby} is devoted to the study of the model of a truncated random variable which is a natural model of  triangular array. This kind of model appears in many proofs of large deviations. Indeed, when one wants to deal with a random variable, the Laplace transform of which is not finite, a classical approach consists in truncating the random variable and in letting the truncation going to infinity. In this model, we exhibit various rate functions, especially nonconvex rate functions.

\section{Main results} \label{sec:main}

For all $n \geqslant 1$, let $\rva{Y}{n}$ be a centered real-valued random variable, let $N_n$ be a natural number, and let $\left(\rva[i]{Y}{n}\right)_{1 \leqslant i \leqslant N_n}$ be a family of i.i.d.\ random variables distributed as $\rva{Y}{n}$. Define, for all $k \in \intervallentff{1}{N_n}$,
\[
\rva[k]{T}{n}\defeq\sum_{i=1}^{k} \rva[i]{Y}{n}.
\]
To lighten notation, let $\rva{T}{n} \defeq \rva[N_n]{T}{n}$.


\begin{thm}[Maximal jump range] \label{thm:nagaev_weak_array_mob_eps}
Let $\varepsilon \in \intervalleoo{0}{1}$, $q > 0$, and $\alpha >(1+\varepsilon)^{-1}$. Assume that:
\begin{description}
\item[(H1)\label{hyp:tails}] 
for all $N_n^{\alpha \varepsilon} \preccurlyeq y_n \preccurlyeq N_n^\alpha$, $\log \Prob(\rva{Y}{n} \geqslant y_n) \sim -q y_n^{1-\varepsilon}$;
\item[(H2)\label{hyp:mt2_weak_array_mob_v2}] 
$\Espe[\rva{Y}{n}^2] = o(N_n^{\alpha(1+\varepsilon)-1})$.
\end{description}
Then, for all $y \geqslant 0$,
\begin{align*}
\lim_{n \to \infty} \frac{1}{N_n^{\alpha(1-\varepsilon)}} \log \Prob(\rva{T}{n} \geqslant N_n^{\alpha} y) = - q y^{1-\varepsilon} .
\end{align*}
\end{thm}

%


%


As in \cite{Nagaev69-1, Nagaev69-2,FATP2020}, the proof of Theorem \ref{thm:nagaev_weak_array_mob_eps} immediately adapts to show that, if $x_n \gg N_n^{(1+\varepsilon)^{-1}}$, if, for all $x_n^{\varepsilon} \leqslant y_n \leqslant x_n(1+\delta)$ for some $\delta>0$, $\log\Prob(\rva{Y}{n} \geqslant y_n)\sim -q y_n^{1-\varepsilon}$, and if $\Var(\rva{Y}{n}) = o(x_n^{(1+\varepsilon)}/N_n)$, then
\begin{align*}
\log \Prob( \rva{T}{n}  \geqslant x_n ) \sim - q x_n^{1-\varepsilon}.
\end{align*}
In this paper (see also Theorems \ref{thm:nagaev_weak_array_mob_eps2} and \ref{thm:nagaev_weak_array_mob_eps_intermediate}), we have chosen to explicit the deviations 
in terms of powers of $N_n$, as it is now standard in large deviation theory. 

\medskip

In addition, the proof of Theorem \ref{thm:nagaev_weak_array_mob_eps} immediately adapts to show that, if $L$ is a slowly varying function 
such that, for all $N_n^{\alpha \varepsilon}/L(N_n^\alpha) \preccurlyeq y_n \preccurlyeq N_n^\alpha$, $\log \Prob(\rva{Y}{n} \geqslant y_n) \sim -L(y_n) y_n^{1-\varepsilon}$ and if assumption \ref{hyp:mt2_weak_array_mob_v2} holds, then,  
for all $y \geqslant 0$,
\begin{align*}
\lim_{n \to \infty} \frac{1}{L(N_n^\alpha) N_n^{\alpha(1-\varepsilon)}} \log \Prob(\rva{T}{n} \geqslant N_n^{\alpha} y) = - y^{1-\varepsilon}.
\end{align*}
The same is true for Theorem \ref{thm:nagaev_weak_array_mob_eps2} below whereas Theorem \ref{thm:nagaev_weak_array_mob_eps_intermediate} below requires additional assumptions on $L$ to take into account the regularly varying tail assumption. 

\medskip

Moreover, if an analogous assumption as \ref{hyp:tails} for the left tail of $\rva{Y}{n}$ is also satisfied, then $\rva{T}{n}$ satisfies a large deviation principle at speed $N_n^{\alpha(1-\varepsilon)}$ with rate function $-q\abs{y}^{1-\varepsilon}$ (the same remark applies to Theorems \ref{thm:nagaev_weak_array_mob_eps2} and \ref{thm:nagaev_weak_array_mob_eps_intermediate}).


\begin{thm}[Gaussian range] \label{thm:nagaev_weak_array_mob_eps2}
Let $\varepsilon \in \intervalleoo{0}{1}$, $q > 0$, and $1/2 < \alpha < (1+\varepsilon)^{-1}$. Suppose that \ref{hyp:tails} holds together with:
\begin{description}
\item[(H2')\label{hyp:var_inf}]
$
\Espe[\rva{Y}{n}^2]\to \sigma^2;
$
\item[(H2+)\label{hyp:mt3_inf}] 
there exists $\gamma \in \intervalleof{0}{1}$ such that $\Espe[\abs{\rva{Y}{n}}^{2+\gamma}]=o(N_n^{\gamma(1-\alpha)})$.
\end{description}

Then, for all $y \geqslant 0$,
\[
\lim_{n \to \infty} \frac{1}{N_n^{2\alpha-1} }\log \Prob( \rva{T}{n} \geqslant N_n^{\alpha} y ) = - \frac{y^2}{2\sigma^2}.
\]
\end{thm}



\begin{thm}[Transition] \label{thm:nagaev_weak_array_mob_eps_intermediate}
Let $\varepsilon \in \intervalleoo{0}{1}$, $q > 0$, and $\alpha = (1+\varepsilon)^{-1}$. Suppose that \ref{hyp:tails}, \ref{hyp:var_inf}, and \ref{hyp:mt3_inf} hold. Then, for all $y \geqslant 0$,
\begin{equation} \label{eq:I_def}
\lim_{n \to \infty} \frac{1}{N_n^{(1-\varepsilon)/(1+\varepsilon)} }\log \Prob( \rva{T}{n} \geqslant N_n^{1/(1+\varepsilon)} y ) = - \inf_{0\leqslant \theta \leqslant 1} \bigl\{ q\theta^{1-\varepsilon} y^{1-\varepsilon}+\frac{(1-\theta)^2y^2}{2\sigma^2} \bigr\} \eqdef - I(y) .
\end{equation}
\end{thm}

Let us explicit a little the rate function $I$. Let  $f(\theta)=q\theta^{1-\varepsilon} y^{1-\varepsilon}+{(1-\theta)^2y^2/}{(2\sigma^2)}$. An easy computation shows that, if $y \leqslant y_0\defeq ((1-\varepsilon^2)(1+1/\varepsilon)^\varepsilon q\sigma^2)^{1/(1+\varepsilon)}$, $f$ is increasing and its minimum $y^2/(2\sigma^2)$ is attained at $\theta=0$. If $y>y_0$,
 $f$ has two local minima, at $0$ and at $\theta(y)$: the latter corresponds to the greatest of the two roots in $\intervalleff{0}{1}$ of $f'(t)=0$, equation equivalent to 
\begin{align} \label{eq:nag_6}
(1-\theta)\theta^{\epsilon}=\frac{(1-\varepsilon)q\sigma^2}{y^{1+\varepsilon}}.
\end{align}
If $y_0< y\leqslant y_1\defeq (1+\varepsilon)\left({q\sigma^2}/{(2\varepsilon)^{\varepsilon}}\right)^{\frac{1}{1+\varepsilon}}$, then $f(\theta(y))\geqslant f(0)$. And if $y>y_1$, $f(\theta(y))< f(0)$. As a consequence, for all $y\geqslant 0$,
\[
I(y)
 = \begin{cases}
\frac{y^2}{2\sigma^2} & \text{if  $y\leqslant y_1$}\\ 
q\theta(y)^{1-\varepsilon} y^{1-\varepsilon}+\frac{(1-\theta(y))^2y^2}{2\sigma^2} & \text{if  $y> y_1$} .
\end{cases}
\]

%
%
%
%
%
%

\section{Proofs} \label{sec:proofs}

\subsection{Proof of Theorem \ref{thm:nagaev_weak_array_mob_eps} (Maximal jump regime)}

Let us fix $y > 0$. The result for $y=0$ follows by monotony. First, we define
\begin{align}
\Prob( \rva{T}{n} \geqslant  N_n^{\alpha} y )
 & = \Prob( \rva{T}{n} \geqslant  N_n^{\alpha} y,\ \forall i \in \intervallentff{1}{N_n} \quad \rva[i]{Y}{n} <  N_n^{\alpha} y) \nonumber \\
 & \hspace{3cm} + \Prob( \rva{T}{n} \geqslant  N_n^{\alpha} y,\ \exists i \in \intervallentff{1}{N_n} \quad \rva[i]{Y}{n} \geqslant  N_n^{\alpha} y) \nonumber \\
 & \eqdef P_{n,0} + R_{n,0} \label{eq:decomp} .
\end{align}
%
%
%
%
%

\begin{lem} \label{lem2_weak_array_mob_v2}
Under \ref{hyp:tails} and \ref{hyp:mt2_weak_array_mob_v2}, for $\alpha > 1/2$ and $y > 0$,
\[
\lim_{n \to \infty} \frac{1}{ N_n^{\alpha(1-\varepsilon)} } \log  R_{n,0} = -q y^{1-\varepsilon} .
\]
\end{lem}

\begin{proof}[Proof of Lemma \ref{lem2_weak_array_mob_v2}]
Using \ref{hyp:tails},
\begin{align}
\limsup_{n \to \infty} \frac{1}{ N_n^{\alpha(1-\varepsilon)} } \log R_{n,0}
 & \leqslant \lim_{n \to \infty} \frac{1}{ N_n^{\alpha(1-\varepsilon)} } \log (N_n \Prob(\rva{Y}{n} \geqslant  N_n^{\alpha} y))
 = - qy^{1-\varepsilon} . \label{eq:R0_maj}
\end{align}
Let us prove the converse inequality. Let $\delta > 0$. We have,
\begin{align*}
R_{n,0} & \geqslant \Prob\left( \rva{T}{n} \geqslant  N_n^{\alpha} y,\ \rva[1]{Y}{n} \geqslant  N_n^{\alpha} y\right)
 \geqslant \Prob\left( \rva[N_n-1]{T}{n} \geqslant -N_n^\alpha\delta\right) \Prob(\rva{Y}{n} \geqslant  N_n^{\alpha} (y +\delta)).
\end{align*}
By Chebyshev's inequality, observe that
\begin{align*}
\Prob( \rva[N_n-1]{T}{n} \geqslant -{N_n}^{\alpha}\delta )
 & \geqslant 1-\frac{\Var(\rva{Y}{n})}{N_n^{2\alpha-1}\delta^2} \to 1 ,
\end{align*}
using \ref{hyp:mt2_weak_array_mob_v2}.
Finally, by \ref{hyp:tails}, one gets
\begin{align*}
\liminf_{n \to \infty} \frac{1}{N_n^{\alpha(1-\varepsilon)}} \log R_{n,0}
 & \geqslant \lim_{n \to \infty} \frac{1}{N_n^{\alpha(1-\varepsilon)}} \log \Prob(\rva{Y}{n} \geqslant  N_n^{\alpha}(y + \delta))
 = - q (y + \delta)^{1-\varepsilon} .
\end{align*}
We conclude by letting $\delta \to 0$.
\end{proof}

To complete the proof of Theorem \ref{thm:nagaev_weak_array_mob_eps}, it remains to prove that, for $\alpha > (1+\varepsilon)^{-1}$,
\begin{equation} \label{lem1_weak_array_mob_v2}
\limsup_{n \to \infty} \frac{1}{ N_n^{\alpha(1-\varepsilon)} }\log P_{n,0} \leqslant -q  y^{1-\varepsilon} ,
\end{equation}
and to apply the principle of the largest term (\emph{see, e.g.}, \cite[Lemma 1.2.15]{DZ98}).
Let $q' \in \intervalleoo{0}{q}$.
Using the fact that $\indic_{x \geqslant 0} \leqslant e^x$, we get
\[
P_{n,0} \leqslant e^{-q' (N_n^{\alpha} y)^{1-\varepsilon}} \Espe\left[ e^{\frac{q'}{ (N_n^{\alpha} y)^{\varepsilon}} \rva{Y}{n}} \indic_{\rva{Y}{n} <  N_n^{\alpha} y} \right]^{N_n} .
\]
If we prove that
\[
\Espe\left[ e^{\frac{q'}{ (N_n^{\alpha} y)^{\varepsilon}} \rva{Y}{n}} \indic_{\rva{Y}{n} <  N_n^{\alpha} y} \right] \leqslant 1 + o ( N_n^{\alpha(1-\varepsilon)-1} ),
\]
then
\[
\log P_{n,0} \leqslant -q' (N_n^{\alpha} y)^{1-\varepsilon} + o(N_n^{\alpha(1-\varepsilon)})
\]
and the conclusion follows by letting $q' \to q$.  Write
\begin{align*}
\Espe &\left[ e^{\frac{q'}{(N_n^{\alpha} y)^{\varepsilon}} \rva{Y}{n}} \indic_{\rva{Y}{n} <  N_n^{\alpha} y} \right]
= \Espe \left[ e^{\frac{q'}{ (N_n^{\alpha} y)^{\varepsilon}} \rva{Y}{n}} \indic_{\rva{Y}{n} <  (N_n^{\alpha} y)^{\varepsilon}} \right] + \Espe \left[ e^{\frac{q'}{ (N_n^{\alpha} y)^{\varepsilon}} \rva{Y}{n}} \indic_{(N_n^{\alpha} y)^{\varepsilon} \leqslant\rva{Y}{n} <  N_n^{\alpha} y} \right].
\end{align*}


First, by a Taylor expansion and \ref{hyp:mt2_weak_array_mob_v2}, we get
\begin{align*}
\Espe \left[ e^{\frac{q'}{ (N_n^{\alpha} y)^{\varepsilon}} \rva{Y}{n}} \indic_{\rva{Y}{n} < (N_n^{\alpha} y)^{\varepsilon}} \right] 
 & \leqslant  \Espe \left[ \Bigl(1 + \frac{q'}{(N_n^{\alpha} y)^{\varepsilon}}\rva{Y}{n} + \frac{(q')^2 e^{q'}}{2 (N_n^{\alpha} y)^{2\varepsilon}}\rva{Y}{n}^2\Bigr) \indic_{\rva{Y}{n} <  (N_n^{\alpha} y)^{\varepsilon}}\right]\\ 
 & \leqslant  1 + \frac{(q')^2 e^{q'}}{2}\cdot \frac{\Espe[\rva{Y}{n}^2]}{(N_n^{\alpha} y)^{2\varepsilon}} \nonumber \\
 & = 1 + o (N_n^{\alpha(1-\varepsilon)-1}) . \nonumber
\end{align*}

To bound above the second expectation, we need the following simple consequence of \ref{hyp:tails}.
\begin{lem} \label{lem:tail}
Under \ref{hyp:tails}, for all $y > 0$,
\[
\forall q' < q \quad \exists n_0 \quad \forall n \geqslant n_0 \quad \forall u \in \intervalleff{(N_n^\alpha y)^\varepsilon}{N_n^\alpha y} \quad \log\Prob(\rva{Y}{n}\geqslant u)\leqslant -q' u^{1-\varepsilon} .
\]
\end{lem}

\begin{proof}[Proof of Lemma \ref{lem:tail}]
By contrapposition, if the conclusion of the lemma is false, we can construct a sequence $(u_n)_{n \geqslant 1}$ such that, for all $n \geqslant 1$, $u_n \in \intervalleff{(N_n^\alpha y)^\varepsilon}{N_n^\alpha y}$ and $\log \Prob(\rva{Y}{n}\geqslant u_n) > -q' u_n^{1-\varepsilon}$, whence \ref{hyp:tails} is not satisfied.
\end{proof}

Now, integrating by parts, we get 
\begin{align*}
\Espe &\left[ e^{\frac{q'}{ (N_n^{\alpha} y)^{\varepsilon}} \rva{Y}{n}} \indic_{(N_n^{\alpha} y)^{\varepsilon} \leqslant\rva{Y}{n} <  N_n^{\alpha} y} \right] 
 = \int_{(N_n^{\alpha} y)^{\varepsilon}}^{N_n^{\alpha} y} e^{\frac{q'}{ (N_n^{\alpha} y)^{\varepsilon}} u}  \Prob(\rva{Y}{n} \in du)\\
 & = - \Big[  e^{\frac{q'}{ (N_n^{\alpha} y)^{\varepsilon}} u} \Prob(\rva{Y}{n} \geqslant u) \Big]_{ (N_n^{\alpha} y)^{\varepsilon}}^{ N_n^{\alpha} y} +  \frac{q'}{ (N_n^{\alpha} y)^{\varepsilon}} \int_{ (N_n^{\alpha} y)^{\varepsilon}}^{ N_n^{\alpha} y} e^{\frac{q'}{ (N_n^{\alpha} y)^{\varepsilon}}u} \Prob(\rva{Y}{n} \geqslant u) du \\
 & \leqslant  e^{q'} \Prob(\rva{Y}{n} \geqslant (N_n^{\alpha} y)^{\varepsilon}) + \frac{q'}{(N_n^{\alpha} y)^{\varepsilon}} \int_{(N_n^{\alpha} y)^{\varepsilon}}^{ N_n^{\alpha} y} e^{\frac{q'}{(N_n^{\alpha} y)^{\varepsilon}} u-q''u^{1-\varepsilon}}du \\
 & \leqslant (1 + q'(N_n^{\alpha} y)^{1-\varepsilon}) e^{q'-q''(N_n^{\alpha} y)^{\varepsilon(1-\varepsilon)}} \\
 & = o(N_n^{\alpha(1-\varepsilon)-1})
\end{align*}
for $n$ large enough, using \ref{hyp:tails} and Lemma \ref{lem:tail} with $q''\in \intervalleoo{q'}{q}$, and the supremum of $u \mapsto q'(N_n^{\alpha} y)^{-\varepsilon} u-q''u^{1-\varepsilon}$ over $\intervalleff{(N_n^{\alpha} y)^{\varepsilon}}{N_n^{\alpha} y}$. The proof of Theorem \ref{thm:nagaev_weak_array_mob_eps} is now complete. 

\subsection{Proof of Theorem \ref{thm:nagaev_weak_array_mob_eps2} (Gaussian regime)}

Let us fix $y > 0$. The result for $y=0$ follows by monotony. For all $m \in \intervallentff{0}{N_n}$, we define
\begin{align*}
\Pi_{l,m}(x)
 & = \Prob\Big( \rva[l]{T}{n} \geqslant  x, \, \forall i \in \intervallentff{1}{l-m} \quad \rva[i]{Y}{n}< (N_n^{\alpha} y)^{\varepsilon},\\
 & \hspace{2cm} \forall i \in \intervallentff{l-m+1}{l} \quad (N_n^{\alpha} y)^{\varepsilon}\leqslant \rva[i]{Y}{n} <  N_n^{\alpha} y \Big) ,
\end{align*}
and we denote $\Pi_{N_n,m}(N_n^{\alpha} y)$ by $\Pi_{n,m}$, so that
\begin{align}\label{eq:decomp2}
P_{n}
 & = P_{n,0}+R_{n,0}
 = \sum_{m=0}^{N_n} \binom{N_n}{m} \Pi_{n,m}+R_{n,0} .
\end{align}
%
%
By Lemma \ref{lem2_weak_array_mob_v2} and the fact that, for $\alpha > (1+\varepsilon)^{-1}$, $2\alpha-1 > \alpha(1-\varepsilon)$, we get
\begin{align} \label{eq:gauss_Rn0}
\lim_{n \to \infty} \frac{1}{ N_n^{2\alpha-1} } \log  R_{n,0} = -\infty .
\end{align}

\begin{lem} \label{lem:pi0} 
Under \ref{hyp:tails}, \ref{hyp:var_inf}, and \ref{hyp:mt3_inf}, for $1/2 < \alpha \leqslant (1+\varepsilon)^{-1}$ and $y > 0$,
\begin{align}\label{eq:pi0}
\lim_{n \to \infty} \frac{1}{N_n^{2\alpha-1}} \log \Pi_{n,0} = - \frac{y^2}{2\sigma^2} .
\end{align}
\end{lem}

\begin{proof}[Proof of Lemma \ref{lem:pi0}] For all $n\geqslant 1$, we introduce the variable $\rva{\Yinf}{n}$ distributed as $\mathcal{L}(\rva{Y}{n}\ |\ \rva{Y}{n} < (N_n^{\alpha}y)^{\varepsilon})$. Let $\rva{\Tinf}{n}=\sum_{i=1}^{N_n} \rva[i]{\Yinf}{n}$ where the $\rva[i]{\Yinf}{n}$ are independent random variables distributed as $\rva{\Yinf}{n}$. Then
\begin{align*}
\Pi_{n,0}=\Prob(\rva{\Tinf}{n} \geqslant N_n^\alpha y) \Prob(\rva{Y}{n}<(N_n^\alpha y)^{\varepsilon})^{N_n} .
\end{align*}
On the one hand, $\Prob(\rva{Y}{n}<(N_n^\alpha y)^{\varepsilon})^{N_n} \to 1$ by \ref{hyp:tails}. On the other hand, in order to apply the unilateral version of Gärtner-Ellis theorem (\emph{see} \cite{PS75}, and \cite{FATP2020} for a modern formulation), we compute, for $u\geqslant 0$, 
\[
\Lambda_n(u)=N_n^{2(1-\alpha)} \log \Espe \left[e^{\frac{u}{N_n^{1-\alpha}}\rva{\Yinf}{n}}\right].
\]

Now, there exists a constant $c>0$ such that, for all $t \leqslant uy^\varepsilon$, $\abs{e^t-(1+t+t^2/2)}\leqslant c\abs{t}^{2+\gamma}$ , whence
\begin{align}\label{eq:taylor2}
\abs{e^{\frac{u}{N_n^{1-\alpha}}\rva{\Yinf}{n}} - 1 - \frac{u}{N_n^{1-\alpha}}\rva{\Yinf}{n} - \frac{u^2}{2N_n^{2(1-\alpha)}}(\rva{\Yinf}{n})^2}
 \leqslant \frac{c u^{2+\gamma}}{N_n^{(2+\gamma)(1-\alpha)}}\abs{\rva{\Yinf}{n}}^{2+\gamma} ,
\end{align}
by the definition of $\rva{\Yinf}{n}$ and $\alpha(1+\varepsilon)\leqslant 1$. Now,
\begin{align}
\left| \Espe \left[e^{\frac{u}{N_n^{1-\alpha}}\rva{\Yinf}{n}} \right] - e^{\frac{u^2 \sigma^2}{2 N_n^{2(1-\alpha)}}} \right|
 & \leqslant \left| \Espe \left[e^{\frac{u}{N_n^{1-\alpha}}\rva{\Yinf}{n}} \right] - \Espe \left[ 1 + \frac{u}{N_n^{1-\alpha}}\rva{\Yinf}{n} + \frac{u^2}{2N_n^{2(1-\alpha)}}(\rva{\Yinf}{n})^2 \right] \right| \nonumber \\
 & \hspace{1cm} + \left| \Espe \left[1 + \frac{u}{N_n^{1-\alpha}}\rva{\Yinf}{n} + \frac{u^2}{2N_n^{2(1-\alpha)}}(\rva{\Yinf}{n})^2 \right] - e^{\frac{u^2 \sigma^2}{2 N_n^{2(1-\alpha)}}} \right|. \label{eq:gaussian_57}
\end{align}
The first term of \eqref{eq:gaussian_57} is bounded above by
\begin{align}
\frac{c u^{2+\gamma}}{N_n^{(2+\gamma)(1-\alpha)}}\Espe[\abs{\rva{\Yinf}{n}}^{2+\gamma}] = o(N_n^{-2(1-\alpha)}), \label{eq:gaussian_2}
\end{align}
by assumptions \ref{hyp:tails} and \ref{hyp:mt3_inf}, and an integration by parts. Using a Taylor expansion of order $2$ of the exponential function, the second term of \eqref{eq:gaussian_57} is equal to
\begin{align}
\left| \frac{u}{N_n^{1-\alpha}} \Espe[\rva{\Yinf}{n}] + \frac{u^2}{2N_n^{2(1-\alpha)}} (\Espe[(\rva{\Yinf}{n})^2]-\sigma^2) + o(N_n^{-2(1-\alpha}) \right|  .
\end{align}
By \ref{hyp:tails} and the fact that $\Espe[\rva{Y}{n}]=0$, $\Espe[\rva{\Yinf}{n}]$ is exponentially decreasing, whence $\Espe[\rva{\Yinf}{n}]=o(1/N_n^{1-\alpha})$; similarly, by \ref{hyp:tails}, \ref{hyp:var_inf}, and  \ref{hyp:mt3_inf}, $\Espe[(\rva{\Yinf}{n})^2]\to \sigma^2$; hence, we get
%
%
%
\begin{align*}
\Lambda_n(u) = \frac{u^2\sigma^2}{2} + o(1) ,
\end{align*}
and the proof of Lemma \ref{lem:pi0} is complete.
\end{proof}

Theorem \ref{thm:nagaev_weak_array_mob_eps2} stems from \eqref{eq:gauss_Rn0}, \eqref{eq:pi0} and the fact that, for $1/2 < \alpha < (1+\varepsilon)^{-1}$,
\begin{equation} \label{eq:pi_sum}
\limsup_{n\to \infty} \frac{1}{N_n^{2\alpha-1}} \log \sum_{m=1}^{N_n} \binom{N_n}{m} \Pi_{n,m}
 \leqslant - \frac{y^2}{2\sigma^2} ,
\end{equation}
the proof of which is given now.
We adapt the proof in \cite[Lemma 5]{Nagaev69-2} and focus on the logarithmic scale. Let $m_n=\lceil N_n^{\alpha (1-\varepsilon)^2 2y^{(1-\varepsilon)^2}}\rceil$. In particular, for all $m > m_n$,
\begin{align}\label{eq:m_n}
m(N_n^{\alpha}y)^{\varepsilon(1-\varepsilon)}\geqslant \frac{m(N_n^{\alpha}y)^{\varepsilon(1-\varepsilon)}}{2} + (N_n^{\alpha}y)^{1-\varepsilon}.
\end{align}

\begin{lem} \label{lem:pinm_gaussian}
Under \ref{hyp:tails}, for $\alpha > 1/2$, $y > 0$, and $q' < q$,
\[
\limsup_{n \to \infty} \frac{1}{N_n^{\alpha(1-\varepsilon)}} \log \sum_{m=m_n+1}^{N_n} \binom{N_n}{m} \Pi_{n,m}
 \leqslant -q'y^{1-\varepsilon} .
\]
\end{lem}

\begin{proof}
For $n$ large enough, using Lemma \ref{lem:tail} and inequality \eqref{eq:m_n},
\begin{align*}
\sum_{m=m_n+1}^{N_n} \binom{N_n}{m}\Pi_{n,m}
&\leqslant \sum_{m=m_n+1}^{N_n} \binom{N_n}{m}\Prob(\forall i\in \intervallentff{1}{m} \quad \rva[i]{Y}{n}\geqslant (N_n^{\alpha}y)^{\varepsilon}) \\
&\leqslant \sum_{m=m_n+1}^{N_n} \binom{N_n}{m}e^{-mq' (N_n^{\alpha}y)^{\varepsilon(1-\varepsilon)}} \\
&\leqslant e^{-q'(N_n^{\alpha}y)^{1-\varepsilon}} 
\sum_{m=m_n+1}^{N_n} \binom{N_n}{m}e^{-mq' (N_n^{\alpha}y)^{\varepsilon(1-\varepsilon)}/2}  \\
&\leqslant e^{-q'(N_n^{\alpha}y)^{1-\varepsilon}} 
\Big(1+e^{-q' (N_n^{\alpha}y)^{\varepsilon(1-\varepsilon)}/2}\Big)^{N_n} .
\end{align*}
\end{proof}

Here, as $\alpha (1-\varepsilon) > 2\alpha-1$, we conclude that
\begin{align}\label{eq:pinmgrand}
\lim_{n\to \infty} \frac{1}{N_n^{2\alpha-1}}\log \sum_{m=m_n+1}^{N_n} \binom{N_n}{m} \Pi_{n,m}
=-\infty.
\end{align}

Now, for $m\in\intervallentff{1}{m_n}$, let us bound above $\Pi_{n,m}$. Let us define
\[
f_m(u_1,\dots,u_m) \defeq \Pi_{N_n-m,0}\biggl( N_n^\alpha y-\sum_{i=1}^m u_i \biggr) ,
\]
which is nondecreasing in each variable. For $q''<q'<q$ and $n$ large enough,
\begin{align*}
\Pi_{n,m}
& = \Prob(\rva{T}{n}\geqslant N_n^\alpha y \, , \, (N_n^\alpha y)^{\varepsilon}\leqslant \rva[1]{Y}{n},\dots,\rva[m]{Y}{n} < N_n^\alpha y \,,\, \rva[m+1]{Y}{n},\dots,\rva[n]{Y}{n}< (N_n^\alpha y)^{\varepsilon}) \\
&= \int_{[(N_n^\alpha y)^{\varepsilon},N_n^\alpha y]^m} f_m(u_1,\dots,u_m)\, d \Prob_{\rva{Y}{n}}(u_1)\cdots d\Prob_{\rva{Y}{n}}(u_m) \\
&= \sum_{(k_1,\dots,k_m)} \int_{\prod_i \intervalleof{k_i-1}{k_i}} f_m(u_1,\dots,u_m)\, d \Prob_{\rva{Y}{n}}(u_1)\cdots d\Prob_{\rva{Y}{n}}(u_m) \\
&\leqslant \sum_{(k_1,\dots,k_m)} f_m(k_1,\dots,k_m) \prod_i  \Prob(\rva{Y}{n}\in \intervalleof{k_i-1}{k_i})\nonumber\\
&\leqslant \sum_{(k_1,\dots,k_m)} f_m(k_1,\dots,k_m) e^{-q'\sum_{i=1}^m (k_i-1)^{1-\varepsilon}} \\
&\leqslant \int_{[(N_n^\alpha y)^{\varepsilon},N_n^\alpha y+2]^m} f_m(u_1,\dots,u_m)\, e^{-q''\sum_{i=1}^m u_i^{1-\varepsilon}}\,du_1\cdots du_m \\
&= I_{1,m}+I_{2,m} ,
\end{align*}
where, for $j\in\{1,2\}$, 
\begin{align}\label{def:Im}
I_{j,m} & \defeq \int_{A_{j,m}} f_m(u_1,\dots,u_m)\, e^{-q'' s_m(u_1, \dots, u_m)}\,du_1\cdots du_m
\end{align}
with
\begin{align*}
A_{1,m}&\defeq \enstq{(u_1,\dots,u_m)\in\intervalleff{(N_n^\alpha y)^{\varepsilon}}{N_n^\alpha y+2}^m}{\sum_{i=1}^m u_i \geqslant N_n^\alpha y},\\
A_{2,m}&\defeq\enstq{(u_1,\dots,u_m)\in\intervalleff{(N_n^\alpha y)^{\varepsilon}}{N_n^\alpha y+2}^m}{\sum_{i=1}^m u_i < N_n^\alpha y},
\end{align*}
and
\[
s_m(u_1, \dots, u_m) \defeq \sum_{i=1}^m u_i^{1-\varepsilon} .
\]

\begin{lem} \label{lem:I1m_gaussian}
For $\alpha > 1/2$ and $y > 0$,
\begin{align*}
\limsup_{n \to \infty} \frac{1}{N_n^{\alpha(1-\varepsilon)}} \log \sum_{m=1}^{m_n} \binom{N_n}{m} I_{1,m}
 \leqslant -q'' y^{1-\varepsilon} .
\end{align*}
\end{lem}

\begin{proof}
Since $s_m$ is concave, $s_m$ reaches its minimum on $A_{1,m}$ at the points with all coordinates equal to $(N_n^\alpha y)^{\varepsilon}$ except one equal to  $N_n^\alpha y-(N_n^\alpha y)^{\varepsilon}(m-1)$. Moreover, using the fact that $f_m(u_1,\dots,u_m)
\leqslant 1$ in \eqref{def:Im}, it follows that, 
for $n$ large enough, for all $m \in \{ 1, \dots, m_n \}$,
\begin{align*}
I_{1,m}
 & \leqslant (N_n^\alpha y)^m e^{-q''(m-1)(N_n^\alpha y)^{\varepsilon(1-\varepsilon)}-q''(N_n^\alpha y-(m-1)(N_n^\alpha y)^{\varepsilon})^{1-\varepsilon}}\\
 & \leqslant(N_n^\alpha y)^m e^{-q''(N_n^\alpha y)^{1-\varepsilon}}e^{-q''(m-1)((N_n^\alpha y)^{\varepsilon(1-\varepsilon)}-1)}.
\end{align*}
Finally,
\begin{align*}
\sum_{m=1}^{m_n} \binom{N_n}{m} I_{1,m}
 & \leqslant e^{-q''(N_n^\alpha y)^{1-\varepsilon}} \sum_{m=1}^{m_n} \binom{N_n}{m} (N_n^\alpha y)^m e^{-q''(m-1)((N_n^\alpha y)^{\varepsilon(1-\varepsilon)}-1)} \\
 & \leqslant e^{-q''(N_n^\alpha y)^{1-\varepsilon}} N_n^{1+\alpha} y \sum_{m=1}^{m_n}  \left(N_n^{1+\alpha} y e^{-q''((N_n^\alpha y)^{\varepsilon(1-\varepsilon)}-1)}\right)^{m-1} , 
\end{align*}
and the conclusion follows, since the latter sum is bounded.
\end{proof}

As $\alpha(1-\varepsilon) > 2\alpha-1$, we conclude that
\begin{equation}
\label{eq:majgausssumI1m}
\lim_{n\to \infty} \frac{1}{N_n^{2\alpha-1}}\log \sum_{m=1}^{m_n} \binom{N_n}{m}I_{1,m}=-\infty. 
\end{equation}

\begin{lem} \label{lem:I2m}
Under \ref{hyp:var_inf} and \ref{hyp:mt3_inf}, for $1/2 < \alpha \leqslant (1+\varepsilon)^{-1}$, $y > 0$, $n$ large enough, and $m \in \{ 1, \dots, m_n \}$,
\[
I_{2,m}
 \leqslant (N_n^\alpha y)^m e^{- q''(m-1)(N_n^\alpha y)^{\varepsilon(1-\varepsilon)}} \exp \biggl( \sup_{m(N_n^\alpha y)^{\varepsilon} \leqslant u < N_n^\alpha y} \phi_m(u) \biggr)
\]
where
\[
\phi_m(u)
 \defeq -\frac{(N_n^\alpha y-u)^2}{2\sigma^2(N_n-m)(1+c_n)} - q'' (u-(m-1)(N_n^\alpha y)^{\varepsilon})^{1-\varepsilon} .
\]
\end{lem}

\begin{proof}
Here, we use Chebyshev's exponential inequality to control $f_m(u_1,\dots,u_m)=\Pi_{N_n-m,0}(N_n^\alpha y - u_1 - \dots - u_m)$ in $I_{2,m}$. For all $l \in \N^*$, for all $x \in \R$, and for all $\lambda \geqslant 0$,
\begin{align*}
\Pi_{l,0}(x)
&\leqslant  
\exp\left\{-\lambda x +l\log \Espe \left[e^{\lambda \rva{Y}{n}} \indic_{\rva{Y}{n} \leqslant (N_n^\alpha y)^{\varepsilon}}\right] \right\}.
\end{align*}

Let $M > y/\sigma^2$. There exists $c > 0$ such that, for all $s \leqslant M$, we have $e^s \leqslant 1 + s + s^2/2+c|s|^{2+\gamma}$. Hence, as soon as $\lambda \leqslant M/N_n^{1-\alpha} \leqslant M/N_n^{\alpha \epsilon}$,
\begin{align}
\Espe \left[ e^{\lambda \rva{Y}{n}} \indic_{\rva{Y}{n} \leqslant (N_n^\alpha y)^\varepsilon} \right]
%
%
 & \leqslant 1 + \frac{\lambda^2}{2} \Espe[\rva{Y}{n}^2] + c \lambda^{2+\gamma} \Espe[\abs{\rva{Y}{n}}^{2+\gamma}]
 \leqslant 1 + \frac{\lambda^2\sigma^2}{2}(1+c_n) , \label{eq:cheb_expo}
\end{align}
where
\[
c_n
 \defeq \Espe[\rva{Y}{n}^2] - \sigma^2 + \frac{2 c M^\gamma}{\sigma^2} N_n^{-\gamma(1-\alpha)} \Espe[\abs{\rva{Y}{n}}^{2+\gamma}]
 = o(1) ,
\]
by \ref{hyp:var_inf} and \ref{hyp:mt3_inf}. Thus, for $\lambda \leqslant M/N_n^{1-\alpha}$,
\[
\Pi_{N_n-m,0}(N_n^\alpha y - u) \leqslant \exp\biggl( -\lambda(N_n^\alpha y - u) + (N_n-m) \frac{\lambda^2 \sigma^2}{2} (1+c_n) \biggr) .
\]
For $n$ large enough and $m \in \{ 1, \dots, m_n \}$, the infimum in $\lambda$ of the last expression is attained at
\[
\lambda^* \defeq \frac{N_n^\alpha y-u}{(N_n-m)\sigma^2(1+c_n)} \leqslant \frac{M}{N_n^{1-\alpha}} ,
\]
and is equal to $-(N_n^\alpha y-u)^2/(2\sigma^2(N_n-m)(1+c_n))$. So, for $n$ large enough:
\begin{equation}
\label{majPin-mgauss}
\Pi_{N_n-m,0}(N_n^\alpha y-u) \leqslant \exp\biggl(-\frac{(N_n^\alpha y-u)^2}{2\sigma^2 (N_n-m)(1+c_n)} \biggr) .
\end{equation}
Since $s_m$ is concave, $s_m$ reaches its minimum on
\[
A_{2,m,u} \defeq \enstq{(u_1,\dots,u_m)\in [(N_n^\alpha y)^{\varepsilon},N_n^\alpha y+2]^m}{\sum_{i=1}^m u_i=u}
\]
at the points with all  coordinates equal to $(N_n^\alpha y)^{\varepsilon}$ except one equal to  $u-(m-1)(N_n^\alpha y)^{\varepsilon}$,
whence, for $n$ large enough and $m \in \{ 1, \dots, m_n \}$,
\begin{align*}
I_{2,m}
 & \leqslant (N_n^\alpha y)^m \sup_{m(N_n^\alpha y)^{\varepsilon}\leqslant u < N_n^\alpha y} \exp \biggl( - \frac{(N_n^\alpha y-u)^2}{2\sigma^2 (N_n-m)(1+c_n)} - q'' (u-(m-1)(N_n^\alpha y)^{\varepsilon})^{1-\varepsilon} \nonumber \\
 & \hspace{9cm} -q''(m-1)(N_n^\alpha y)^{\varepsilon(1-\varepsilon)} \biggr) \\
 & \leqslant (N_n^\alpha y)^m e^{- q''(m-1)(N_n^\alpha y)^{\varepsilon(1-\varepsilon)}} \exp \biggl( \sup_{m(N_n^\alpha y)^{\varepsilon} \leqslant u < N_n^\alpha y} \phi_m(u) \biggr) . \nonumber
\end{align*}    
\end{proof}



Now, for $1/2 < \alpha < (1+\varepsilon)^{-1}$, $n$ large enough, and $m \in \{ 1, \dots, m_n \}$, the function $\phi_m$ is decreasing on $\intervallefo{m(N_n^\alpha y)^{\varepsilon}}{N_n^\alpha y}$. So,
\begin{align*}
I_{2,m}
 & \leqslant (N_n^\alpha y)^m \exp\left(-q''m(N_n^\alpha y)^{\varepsilon(1-\varepsilon)}-\frac{(N_n^\alpha y-m_n(N_n^\alpha y)^{\varepsilon})^2}{2N_n\sigma^2(1+c_n)} \right).
\end{align*}

It follows that
\[
\sum_{m=1}^{m_n} \binom{N_n}{m} I_{2,m}
\leqslant e^{-\frac{(N_n^\alpha y-m_n(N_n^\alpha y)^{\varepsilon})^2}{2N_n\sigma^2(1+c_n)}} \sum_{m=1}^{m_n} \left(N_n^{1+\alpha} y e^{-q''(N_n^\alpha y)^{\varepsilon(1-\varepsilon)}}\right)^m.
\]
Since the latter sum is bounded, we get
\begin{align}
\label{eq:majgausssumI2m}
\limsup_{n\to \infty} \frac{1}{N_n^{2\alpha-1}}\log \sum_{m=1}^{m_n} \binom{N_n}{m}I_{2,m}
&\leqslant -\frac{y^2}{2\sigma^2}
\end{align}
as $m_n(N_n^\alpha y)^{\varepsilon}=o(N_n^\alpha)$ and $c_n=o(1)$. 
By \eqref{eq:pinmgrand}, \eqref{eq:majgausssumI1m}, and \eqref{eq:majgausssumI2m}, we get the required result.

\begin{rem}
Notice that, using the contraction principle, one can show that, for all fixed $m$,
\begin{align*}
\limsup_{n\to\infty} \frac{1}{N_n^{(1-\varepsilon)/(1+\varepsilon)}}\log \Pi_{n,m} = 
-\frac{y^2}{2\sigma^2}.
\end{align*}
\end{rem}

\subsection{Proof of Theorem \ref{thm:nagaev_weak_array_mob_eps_intermediate} (Transition)}

Here, we assume \ref{hyp:tails}, \ref{hyp:var_inf}, and \ref{hyp:mt3_inf}, and we deal with the case $\alpha = (1+\varepsilon)^{-1}$, so that $\alpha(1-\varepsilon) = 2\alpha-1 = (1-\varepsilon)/(1+\varepsilon)$. Let us fix $y > 0$. The result for $y=0$ follows by monotony. We still consider the decomposition \eqref{eq:decomp2}. By Lemmas \ref{lem2_weak_array_mob_v2} and \ref{lem:pi0}, and the very definition of $I$ in \eqref{eq:I_def}, we have
\[
\lim_{n \to \infty} \frac{1}{N_n^{(1-\varepsilon)/(1+\varepsilon)}} \log  R_{n,0} = -q y^{1-\varepsilon} \leqslant -I(y)
\]
and
\[
\lim_{n \to \infty} \frac{1}{N_n^{(1-\varepsilon)/(1+\varepsilon)}} \log  \Pi_{n,0} = -\frac{y^2}{2\sigma^2} \leqslant -I(y) .
\]
To complete the proof of Theorem \ref{thm:nagaev_weak_array_mob_eps_intermediate}, it remains to prove that
\begin{align}
& \liminf_{n\to\infty} \frac{1}{N_n^{(1-\varepsilon)/(1+\varepsilon)}} \log \Pi_{n,1} \geqslant -I(y) \label{eq:pi1} \\
& \limsup_{n\to\infty} \frac{1}{N_n^{(1-\varepsilon)/(1+\varepsilon)}} \log \sum_{m=1}^{N_n}\Pi_{n,m} \leqslant -I(y) \label{eq:pi_sum_intermediate}
\end{align}
and to apply the principle of the largest term.

\begin{proof}[Proof of \eqref{eq:pi1}]
For all $t \in \intervalleoo{0}{1}$,
\begin{align*}
\frac{1}{N_n^{(1-\varepsilon)/(1+\varepsilon)}} \log \Pi_{n,1}
 & \geqslant \frac{1}{N_n^{(1-\varepsilon)/(1+\varepsilon)}} \log \Prob\bigl(  \rva[N_n-1]{T}{n} \geqslant N_n^\alpha ty,\ \forall i \in \intervallentff{1}{N_n-1} \quad \rva[i]{Y}{n} < (N_n^\alpha y)^\varepsilon \bigr) \\
 & \hspace{2cm} + \frac{1}{N_n^{(1-\varepsilon)/(1+\varepsilon)}} \log \Prob\bigl(N_n^\alpha (1-t)y \leqslant \rva[N_n]{Y}{n} < N_n^\alpha y \bigr) \\
 & \xrightarrow[n \to \infty]{} - \frac{t^2y^2}{2\sigma^2} - q(1-t)^{1-\varepsilon} y^{1-\varepsilon} ,
\end{align*}
by Lemma \ref{lem:pi0} and by \ref{hyp:tails}. Optimizing in $t \in \intervalleoo{0}{1}$ provides the conclusion.
\end{proof}

\begin{proof}[Proof of \eqref{eq:pi_sum_intermediate}]
We follow the same lines as in the proof of \eqref{eq:pi_sum}. By Lemma \ref{lem:pinm_gaussian}, letting $q'\to q$, we get
\begin{align} \label{eq:pinmgrand_intermediate}
\limsup_{n\to \infty} \frac{1}{N_n^{(1-\varepsilon)/(1+\varepsilon)}} \log \sum_{m=m_n+1}^{N_n} \binom{N_n}{m} \Pi_{n,m}
 \leqslant -q y^{1-\varepsilon}
 \leqslant -I(y) .
\end{align}
Let $\eta = 1-q''/q \in \intervalleoo{0}{1}$. By Lemma \ref{lem:I1m_gaussian},
\begin{align} \label{eq:majgausssumI1m_intermediate}
\limsup_{n\to \infty} \frac{1}{N_n^{(1-\varepsilon)/(1+\varepsilon)}} \log \sum_{m=1}^{m_n} \binom{N_n}{m} I_{1,m}
 \leqslant - q'' y^{1-\varepsilon}
 \leqslant - (1-\eta) I(y) .
\end{align}
Now, recall that Lemma \ref{lem:I2m} provides $I_{2,m} \leqslant (N_n^\alpha y)^m e^{- q''(m-1)(N_n^\alpha y)^{\varepsilon(1-\varepsilon)}} e^{M_n}$, for $n$ large enough, where
\begin{align*}
M_n
 & = \sup_{m(N_n^\alpha y)^{\varepsilon} \leqslant u < N_n^\alpha y} \biggl( -\frac{(N_n^\alpha y-u)^2}{2\sigma^2(N_n-m)(1+c_n)} - q'' (u-(m-1)(N_n^\alpha y)^{\varepsilon})^{1-\varepsilon} \biggr) \\
 & \leqslant N_n^{(1-\varepsilon)/(1+\varepsilon)} \sup_{m(N_n^\alpha y)^{-1+\varepsilon} \leqslant \theta < 1} \biggl( -\frac{(1-\theta)^2 y^2}{2\sigma^2 (1+c_n)} - q'' \theta^{1-\varepsilon} y^{1-\varepsilon} \biggl(1 - \frac{(m-1)(N_n^\alpha y)^{-1+\varepsilon}}{\theta} \biggr)^{1-\varepsilon} \biggr).
\end{align*}
For $n$ large enough, for all $m \in \{ 1, \dots, m_n \}$,
\begin{align*}
\inf_{m(N_n^\alpha y)^{-1+\varepsilon} \leqslant \theta < \eta} \left\{ \frac{(1-\theta)^2 y^2}{2\sigma^2(1+c_n)} + q'' \theta^{1-\epsilon} \left( 1 -\frac{(m-1)(N_n^\alpha y)^{-1+\varepsilon}}{\theta} \right)^{1-\epsilon} \right\}
 & \geqslant \frac{(1-\eta)^2 y^2}{2\sigma^2} (1-\eta) \\
 & \geqslant (1-\eta)^3 I(y)
\end{align*}
and
\begin{align*}
 & \inf_{\eta \leqslant \theta < 1} \left\{ \frac{(1-\theta)^2 y^2}{2\sigma^2(1+c_n)} + q'' \theta^{1-\epsilon} \left( 1 -\frac{(m-1)(N_n^\alpha y)^{-1+\varepsilon}}{\theta} \right)^{1-\epsilon} \right\} \\
 \geqslant & \inf_{\eta \leqslant \theta < 1} \left\{ \frac{(1-\theta)^2 y^2}{2\sigma^2} + q'' \theta^{1-\epsilon} y^{1-\varepsilon} \right\} (1-\eta) \\
 \geqslant & \mathop{} (1-\eta)^2 I(y) .
\end{align*}
So $M_n \leqslant - N_n^{(1-\varepsilon)/(1+\varepsilon)} (1-\eta)^3 I(y)$ and
\begin{align}
 & \limsup_{n \to \infty} \frac{1}{N_n^{(1-\varepsilon)/(1+\varepsilon)}} \log \sum_{m=1}^{m_n} \binom{n}{m} I_{2,m} \nonumber \\
 \leqslant & - (1-\eta)^3 I(y) + \limsup_{n \to \infty} \frac{1}{N_n^{(1-\varepsilon)/(1+\varepsilon)}} \log \sum_{m=1}^{m_n} \binom{n}{m} \bigl( N_n^\alpha y e^{- q''(N_n^\alpha y)^{\varepsilon(1-\varepsilon)}} \bigr)^m \nonumber \\
 = & - (1-\eta)^3 I(y) . \label{eq:I2m_transition}
\end{align}
Finally, \eqref{eq:pinmgrand_intermediate}, \eqref{eq:majgausssumI1m_intermediate}, and \eqref{eq:I2m_transition} imply
\[
\limsup_{n\to\infty} \frac{1}{N_n^{(1-\varepsilon)/(1+\varepsilon)}} \log \sum_{m=1}^{N_n}\Pi_{n,m} \leqslant - (1-\eta)^3 I(y) ,
\]
and \eqref{eq:pi_sum_intermediate} follows, letting $q'' \to q$, i.e.\ $\eta \to 0$.
%
%
%
\end{proof}

\begin{rem}
Notice that, using the contraction principle, one can show that, for all fixed $m$,
\begin{align*}
\limsup_{n\to\infty} \frac{1}{N_n^{(1-\varepsilon)/(1+\varepsilon)}}\log \Pi_{n,m} = 
-I(y).
\end{align*}
\end{rem}

\section{About the assumptions}\label{sec:assump}



Looking into the proof of Theorem \ref{thm:nagaev_weak_array_mob_eps}, one can see that assumption \ref{hyp:tails} can be weakened and one may only assume the two conditions that follow.
\begin{thm}
The conclusion of Theorem \ref{thm:nagaev_weak_array_mob_eps} holds under \ref{hyp:mt2_weak_array_mob_v2} and:
\begin{description}
\item[(H1a)\label{hyp:h1a}] for all $y_n = \Theta(N_n^\alpha)$, $\log \Prob(\rva{Y}{n} \geqslant y_n) \sim -q y_n^{1-\varepsilon}$;
\item[(H1b)\label{hyp:h1b}] for all $N_n^{\alpha \varepsilon} \preccurlyeq y_n \preccurlyeq N_n^\alpha$, $\limsup y_n^{-(1-\varepsilon)} \log \Prob(\rva{Y}{n} \geqslant y_n) \leqslant -q$.
\end{description}
\end{thm}

\begin{lem}
\ref{hyp:h1a} is equivalent to:
\begin{description}
\item[(H1a')\label{hyp:h1a'}] for all $y > 0$, $\log \Prob(\rva{Y}{n} \geqslant N_n^\alpha y) \sim - q (N_n^\alpha y)^{1-\varepsilon}$.
\end{description}
\end{lem}

\begin{proof}
If $ N_n^\alpha c_1\leqslant y_n \leqslant N_n^\alpha c_2$, then
\[
-qc_2 \leqslant N_n^{-\alpha(1-\varepsilon)} \log \Prob(\rva{Y}{n} \geqslant y_n) \leqslant -q c_1 .
\]
First extract a convergent subsequence; then, again extract a subsequence such that $N_n^{-\alpha} y_n$ is convergent and use \ref{hyp:h1a} to show that $N_n^{-\alpha(1-\varepsilon)} \log \Prob(\rva{Y}{n} \geqslant y_n)$ is convergent.
\end{proof}

\medskip

\begin{lem}
\ref{hyp:h1b} is equivalent to the conclusion of Lemma \ref{lem:tail}:
\begin{description}
\item[(H1b')\label{hyp:h1b'}] $\forall y > 0 \quad \forall q' < q \quad \exists n_0 \quad \forall n \geqslant n_0 \quad \forall u \in \intervalleff{(N_n^\alpha y)^\varepsilon}{N_n^\alpha y} \quad \log\Prob(\rva{Y}{n}\geqslant u)\leqslant -q' u^{1-\varepsilon}$.
\end{description}
\end{lem}

\begin{proof}
See the proof of Lemma \ref{lem:tail}.
\end{proof}

\begin{thm}
The conclusion of Theorem \ref{thm:nagaev_weak_array_mob_eps} holds under assumptions  \ref{hyp:h1a}, \ref{hyp:h1b}, and 
$\Espe[\abs{\rva{Y}{n}}^{2+\gamma}]/\Espe[\abs{\rva{Y}{n}}^{2}]^{1+\gamma/2}=o(N_n^{\gamma/2})$.
%
\end{thm}

\begin{proof}
The only modification in the proof is the minoration of $R_{n,0}$:
\[
R_{n,0} \geqslant \Prob(\rva[N_n-1]{T}{n} \geqslant 0) \Prob(\rva{Y}{n} \geqslant  N_n^{\alpha} y) .
\]
Now Lyapunov's theorem \cite[Theorem 27.3]{billingsley2013convergence} applies and provides $\Prob(\rva[N_n-1]{T}{n} \geqslant 0) \to 1/2$.
\end{proof}

As for Theorem \ref{thm:nagaev_weak_array_mob_eps2}, assumption \ref{hyp:tails} can be weakened and one may only assume \ref{hyp:h1b}, or even the following weaker assumption.
\begin{thm}
The conclusion of Theorem \ref{thm:nagaev_weak_array_mob_eps2} holds under \ref{hyp:var_inf}, \ref{hyp:mt3_inf}, and:
\begin{description}
\item[(H1c)\label{hyp:h1c}] $\forall y > 0 \quad \exists q > 0 \quad \exists n_0 \quad \forall n \geqslant n_0 \quad \forall u \in \intervalleff{(N_n^\alpha y)^\varepsilon}{N_n^\alpha y} \quad \log \Prob(\rva{Y}{n} \geqslant u) \leqslant - q u^{1-\varepsilon}$.
\end{description}
\end{thm}

Finally, in Theorem \ref{thm:nagaev_weak_array_mob_eps_intermediate}, assumption \ref{hyp:tails} can be weakened and one may only assume \ref{hyp:h1a} and \ref{hyp:h1b}.

\section{Application: truncated random variable}\label{sec:baby}

Let us consider a centered real-valued random variable $Y$, admitting a finite moment of order $2+\gamma$ for some $\gamma>0$. Set $\sigma^2 \defeq \Espe[Y^2]$. Now, let $\beta > 0$ and $c > 0$. For all $n \geqslant 1$, let us introduce the truncated random variable $\rva{Y}{n}$ defined by $\mathcal{L}(\rva{Y}{n}) = \mathcal{L}(Y\ |\ Y <  N_n^\beta c)$. Such truncated random variables naturally appear in proofs of large deviation results. 

\medskip

If $Y$ has a light-tailed distribution, i.e.\ $\Lambda_Y(\lambda)\defeq \log \Espe[e^{\lambda Y}]<\infty$ for some $\lambda>0$, then (the unilateral version of) Gärtner-Ellis theorem applies:
\begin{itemize}
\item if $\alpha\in \intervalleoo{1/2}{1}$, then 
\[
\lim_{n\to \infty} \frac{1}{N_n^{2\alpha-1}} \log \Prob(\rva{T}{n} \geqslant N_n^\alpha y) = -\frac{y^2}{2\sigma^2};
\]
\item if $\alpha=1$, then 
\[
\lim_{n\to \infty} \frac{1}{N_n} \log \Prob(\rva{T}{n} \geqslant N_n^\alpha y) = -\Lambda_Y^*(y)\defeq -\sup_{\lambda\geqslant 0}\{\lambda y-\Lambda_Y(\lambda)\}.
\]
\end{itemize}
Note that we recover the same asymptotics as for the non truncated random variable $Y$. In other words, the truncation does not impact the deviation behaviour. 

\medskip

Now we consider the case where 
$\log \Prob(Y \geqslant y) \sim -q y^{1-\varepsilon}$
for some $q>0$ and $\varepsilon \in \intervalleoo{0}{1}$. In this case, Gärtner-Ellis theorem does not apply since all the rate functions are not convex as usual (as can be seen in Figures \ref{fig:fonction_taux_1} to \ref{fig:fonction_taux_3}). 
Observe that, as soon as $y_n \to \infty$,
\[
\limsup_{n \to \infty} \frac{1}{y_n^{1-\varepsilon}} \log \Prob(\rva{Y}{n} \geqslant y_n) = \limsup_{n\to \infty} \frac{1}{y_n^{1-\varepsilon}} \left(\log \Prob (y_n \leqslant Y < N_n^\beta c)-\log \Prob (Y < N_n^\beta c)\right) \leqslant -q ,
\]
so \ref{hyp:h1b} is satisfied. If, moreover, $y_n \leqslant  N_n^\beta c'$ with $c' < c$, then $\log \Prob(\rva{Y}{n} \geqslant y_n) \sim -q y_n^{1-\varepsilon}$, so \ref{hyp:tails} is satisfied for $\alpha < \beta$. In addition, $\Espe[\rva{Y}{n}]-\Espe[Y]$, $\Espe[\rva{Y}{n}^2]-\Espe[Y^2]$ and $\Espe[\rva{Y}{n}^{2+\gamma}]-\Espe[Y^{2+\gamma}]$ are exponentially decreasing to zero. Therefore, our theorems directly apply for $\alpha < \max(\beta, (1+\varepsilon)^{-1})$, and even for $\alpha = (1+\varepsilon)^{-1} < \beta$. For $\alpha \geqslant \max(\beta, (1+\varepsilon)^{-1})$, the proofs easily adapt to cover all cases. To expose the results, we separate the three cases $\beta > (1+\varepsilon)^{-1}$, $\beta < (1+\varepsilon)^{-1}$ and $\beta = (1+\varepsilon)^{-1}$ and provide a synthetic diagram at the end of the section (page \pageref{graph:alpha_beta}) and the graphs of the exhibited rate functions (pages \pageref{fig:fonction_taux_1} and \pageref{fig:fonction_taux_3}).

\subsection{Case \texorpdfstring{$\beta > (1+\varepsilon)^{-1}$}{f}}

\paragraph{Gaussian range}
When $\alpha < (1+\epsilon)^{-1}$, Theorem \ref{thm:nagaev_weak_array_mob_eps2} applies and, for all $y \geqslant 0$,
\[
\lim_{n \to \infty} \frac{1}{N_n^{2\alpha-1}} \log \Prob(\rva{T}{n} \geqslant N_n^{\alpha} y) = - \frac{y^2}{2\sigma^2} .
\]

\paragraph{Transition 1}
When $\alpha = (1+\epsilon)^{-1}$, Theorem \ref{thm:nagaev_weak_array_mob_eps_intermediate} applies and, for all $y \geqslant 0$,
\[
\lim_{n \to \infty} \frac{1}{N_n^{(1-\varepsilon)/(1+\varepsilon)}} \log \Prob(\rva{T}{n} \geqslant N_n^{\alpha} y) = - I_1(y) \defeq - I(y)= - \inf_{0\leqslant \theta \leqslant 1} \bigl\{ q\theta^{1-\varepsilon} y^{1-\varepsilon}+\frac{(1-\theta)^2y^2}{2\sigma^2} \bigr\} .
\]

\paragraph{Maximal jump range}
When $(1+\epsilon)^{-1} < \alpha < \beta$, Theorem \ref{thm:nagaev_weak_array_mob_eps} applies and, for all $y \geqslant 0$,
\[
\lim_{n \to \infty} \frac{1}{N_n^{\alpha(1-\varepsilon)}} \log \Prob(\rva{T}{n} \geqslant N_n^{\alpha} y) = - q y^{1-\varepsilon} .
\]

\paragraph{Transition 2}
When $\alpha = \beta$, for all $y \geqslant 0$,
\[
\lim_{n \to \infty} \frac{1}{N_n^{\alpha(1-\varepsilon)}} \log \Prob(\rva{T}{n} \geqslant N_n^{\alpha} y)= - I_2(y) \defeq - q\left(\floor{y/c} c^{1-\varepsilon} + (y - \floor{y/c} c)^{1-\varepsilon} \right).
\]

Here, as in all cases where $\alpha \geqslant \beta$, we adapt the definitions \eqref{eq:decomp} and \eqref{eq:decomp2} as:
\begin{equation} \label{eq:decomp3}
\Prob( \rva{T}{n} \geqslant  N_n^{\alpha} y )
 = \Prob( \rva{T}{n} \geqslant  N_n^{\alpha} y,\ \forall i \in \intervallentff{1}{N_n} \quad \rva[i]{Y}{n} <  N_n^{\beta} c)
 \eqdef P_{n,0}
\end{equation}
($R_{n,0} = 0$) and, for all $m \in \intervallentff{0}{N_n}$,
\begin{align}
\Pi_{n,m}
 & = \Prob\Big( \rva{T}{n} \geqslant  N_n^{\alpha} y, \, \forall i \in \intervallentff{1}{N_n-m} \quad \rva[i]{Y}{n}< (N_n^{\beta}c)^\varepsilon, \nonumber\\
 & \hspace{2cm} \forall i \in \intervallentff{N_n-m+1}{N_n} \quad (N_n^{\beta}c)^\varepsilon \leqslant \rva[i]{Y}{n} <  N_n^{\beta} c \Big) . \label{eq:decomp4}
\end{align}
For all $t > 0$,
\begin{align*}
\Pi_{n,0}&= \Prob(\rva{T}{n} \geqslant N_n^\alpha y,\ \forall i \in \intervallentff{1}{N_n} \quad \rva[i]{Y}{n} < (N_n^{\alpha}c)^\varepsilon)\\
 & \leqslant e^{-tyN_n^{\alpha(1-\varepsilon)}} \Espe\bigl[ e^{t N_n^{-\alpha\varepsilon} \rva{Y}{n}} \indic_{\rva{Y}{n} < (N_n^{\alpha}c)^\varepsilon} \bigr]^{N_n} \\
 & = e^{-tyN_n^{\alpha(1-\varepsilon)}(1+o(1))} ,
\end{align*}
(see the proof of Theorem \ref{thm:nagaev_weak_array_mob_eps}), whence Lemma \ref{lem:pi0} with $\mathcal{L}(\rva{\Yinf}{n}) = \mathcal{L}(\rva{Y}{n}\ |\ \rva{Y}{n} < N_n^\alpha c)$ updates into
\[
%
\frac{1}{N_n^{\alpha(1-\varepsilon)}} \log \Pi_{n,0} \xrightarrow[n \to \infty]{} -\infty .
\]
So, by the contraction principle, for all fixed $m \geqslant 0$,
\[
\frac{1}{N_n^{\alpha(1-\varepsilon)}} \log \Pi_{n,m}
 \xrightarrow[n \to \infty]{} \begin{cases}
-\infty & \text{if $m \leqslant y/c-1$} \\
-q\left(\floor{y/c} c^{1-\varepsilon} + (y - \floor{y/c} c)^{1-\varepsilon} \right) & \text{otherwise,}
\end{cases}
\]
that provides a minoration of the sum of the  $\Pi_{n,m}$'s. To obtain a majoration, let us introduce $m_n=\lceil N_n^{\alpha (1-\varepsilon)^2} 2k\rceil$ where $k=\floor{y/c} c^{1-\varepsilon} + (y - \floor{y/c} c)^{1-\varepsilon}$.
Lemma \ref{lem:pinm_gaussian} remains unchanged while Lemmas \ref{lem:I1m_gaussian} and \ref{lem:I2m} requires adjustments. The integration domains defining $I_{1,m}$ and $I_{2,m}$ become
\begin{align*}
A_{1,m}&=\enstq{(u_1,\dots,u_m)\in\intervalleff{(N_n^{\alpha}c)^\varepsilon}{N_n^\alpha c+2}^m}{\sum_{i=1}^m u_i \geqslant N_n^\alpha y},\\
A_{2,m}&=\enstq{(u_1,\dots,u_m)\in\intervalleff{(N_n^{\alpha}c)^\varepsilon}{N_n^\alpha c+2}^m}{\sum_{i=1}^m u_i < N_n^\alpha y},
\end{align*}
Further, the concave function $s_m$ attains its minimum at points with all coordinates equal to $(N_n^{\alpha}c)^\varepsilon$ except $\lfloor y/c_n \rfloor$ coordinates equal to $N_n^\alpha c_n$ and one coordinate equal to  $N_n^\alpha\left( y-\lfloor y/c_n \rfloor c_n\right) -(N_n^{\alpha}c)^\varepsilon \left(m-1-\lfloor y/c_n \rfloor \right)$ with $c_n=c+2N_n^{-\alpha}$. Then following the same lines as in the proof of Lemmas \ref{lem:I1m_gaussian} and \ref{lem:I2m}, we get, for $j\in\{1,2\}$,
\begin{align*}
&\lim_{n\to \infty} \frac{1}{N_n^{2\alpha-1}}\log \sum_{m=1}^{m_n} \binom{N_n}{m}I_{j,m}=-q\left(\floor{y/c} c^{1-\varepsilon} + (y - \floor{y/c} c)^{1-\varepsilon} \right).
\end{align*}

\paragraph{Truncated maximal jump range}
When $\beta < \alpha < \beta + 1$ and $y \geqslant 0$, or $\alpha = \beta+1$ and $y < c$, the proof of Theorem \ref{thm:nagaev_weak_array_mob_eps} adapts and provides
\[
\lim_{n \to \infty} \frac{1}{N_n^{\alpha-\beta\varepsilon}} \log \Prob(\rva{T}{n} \geqslant N_n^{\alpha} y) = - q y c^{-\varepsilon} .
\]
As in the previous case, we use the decomposition given by \eqref{eq:decomp3} and \eqref{eq:decomp4}. To upper bound $P_{n,0}$, we write
\[
P_{n,0} \leqslant e^{- qy c^{-\varepsilon} N_n^{\alpha-\beta\varepsilon}} \Espe\left[ e^{y c^{-\varepsilon} N_n^{-\beta\varepsilon} \rva{Y}{n}} \bigl(\indic_{\rva{Y}{n} < (N_n^\beta c)^\varepsilon} + \indic_{(N_n^\beta c)^\varepsilon \leqslant \rva{Y}{n} <  N_n^\beta c}\bigr) \right]^{N_n}
\]
and follow the same lines as in the proof of Theorem \ref{thm:nagaev_weak_array_mob_eps}. To lower bound $P_{n,0}$, we write, for $c' < c$,
\begin{align*}
\log P_{n,0}
 & \geqslant \log \Prob(\forall i \in \llbracket 1, \lceil N_n^{\alpha-\beta} y/c' \rceil \rrbracket \quad \rva[i]{Y}{n} \geqslant N_n^\beta c') \\
 & \sim - N_n^{\alpha-\beta} y (c')^{-1} q (N_n^\beta c' )^{1-\varepsilon} \\
 & = - N_n^{\alpha-\beta\varepsilon} q y (c')^{-\varepsilon} ,
\end{align*}
and we recover the upper bound, when $c' \to c$.

\paragraph{Trivial case}
When $\alpha=\beta+1$ and $y \geqslant c$, or $\alpha > \beta+1$, 
we obviously have $\Prob(\rva{T}{n} \geqslant N_n^\alpha y) = 0$.

\subsection{Case \texorpdfstring{$\beta < (1+\varepsilon)^{-1}$}{f}}

Here, Theorem \ref{thm:nagaev_weak_array_mob_eps2} applies for $\alpha < (1+\varepsilon)^{-1}$. The notable fact is that the Gaussian range is extended: it spreads until $\alpha < 1-\beta\varepsilon$. 

\paragraph{Gaussian range}
When $\alpha < 1-\beta\varepsilon$, the proof of Theorem \ref{thm:nagaev_weak_array_mob_eps2} adapts and, for all $y \geqslant 0$,
\[
\lim_{n \to \infty} \frac{1}{N_n^{2\alpha-1}} \log \Prob(\rva{T}{n} \geqslant N_n^{\alpha} y) = - \frac{y^2}{2\sigma^2} .
\]
As we said, the result for $\alpha < (1+\varepsilon)^{-1}$ is a consequence of Theorem \ref{thm:nagaev_weak_array_mob_eps2}. Now, suppose $\alpha \geqslant (1+\varepsilon)^{-1} > \beta$. We use the decomposition given by \eqref{eq:decomp3} and \eqref{eq:decomp4}.
Lemma \ref{lem:pi0} works for $\alpha < 1-\beta\varepsilon$, with $\mathcal{L}(\rva{\Yinf}{n}) = \mathcal{L}(\rva{Y}{n}\ |\ \rva{Y}{n} < (N_n^\beta c)^\varepsilon)$. Then, we choose $m_n = \lceil  N_n^{\alpha-2\beta\varepsilon+\beta\varepsilon^2} 2 y c^{-\varepsilon} \rceil$. We obtain the equivalent of Lemma \ref{lem:pinm_gaussian}:
\[
\limsup_{n \to \infty} \frac{1}{N_n^{\alpha-\beta\varepsilon}} \log \sum_{m=m_n+1}^{N_n} \binom{N_n}{m} \Pi_{n,m}
 \leqslant - q' y c^{-\varepsilon} .
\]
with $N_n^{\alpha-\beta\varepsilon} \gg N_n^{2\alpha-1}$. Finally, Lemmas \ref{lem:I1m_gaussian} and \ref{lem:I2m} adapt as well, with
\begin{align*}
A_{1,m}&=\enstq{(u_1,\dots,u_m)\in\intervalleff{(N_n^\beta c)^{\varepsilon}}{N_n^\beta c+2}^m}{\sum_{i=1}^m u_i \geqslant N_n^{\alpha} y },\\
A_{2,m}&=\enstq{(u_1,\dots,u_m)\in\intervalleff{(N_n^\beta c)^{\varepsilon}}{N_n^\beta c+2}^m}{\sum_{i=1}^m u_i <  N_n^{\alpha} y} .
\end{align*}


\paragraph{Transition 3}
When $\alpha = 1-\beta\varepsilon$, the proof of Theorem \ref{thm:nagaev_weak_array_mob_eps_intermediate} adapts and, for all $y \geqslant 0$,
\begin{align*}
\lim_{n \to \infty} \frac{1}{N_n^{1-2\beta\varepsilon}} \log \Prob(\rva{T}{n} \geqslant N_n^{1-\beta\varepsilon} y)
 = -I_3(y) &\defeq - \inf_{0 \leqslant t \leqslant 1} \Bigl\{ q(1-t) y c^{-\varepsilon} + \frac{t^2 y^2}{2\sigma^2} \Bigr\} \\
 & = - \begin{cases}
\frac{y^2}{2\sigma^2} & \text{if $y \leqslant y_3$} \\
\frac{qy}{c^\varepsilon} - \frac{q^2\sigma^2}{2c^{2\varepsilon}} & \text{if $y > y_3$}
\end{cases}
\end{align*}
with $y_3\defeq q\sigma^2 c^{-\varepsilon}$.

\paragraph{Truncated maximal jump range}
When $1-\beta\varepsilon < \alpha < 1+\beta$ and $y \geqslant 0$, or $\alpha = 1+\beta$ and $y < c$, as before, the proof of Theorem \ref{thm:nagaev_weak_array_mob_eps} adapts and
\[
\lim_{n \to \infty} \frac{1}{N_n^{\alpha-\beta\varepsilon}} \log \Prob(\rva{T}{n} \geqslant N_n^{\alpha} y) = - q y c^{-\varepsilon} .
\]

\paragraph{Trivial case}
When $\alpha=\beta+1$ and $y \geqslant c$, or $\alpha > \beta+1$, 
we obviously have $\Prob(\rva{T}{n} \geqslant N_n^\alpha y) = 0$.

\subsection{Case \texorpdfstring{$\beta = (1+\varepsilon)^{-1}$}{f}}

\paragraph{Gaussian range}
When $\alpha < (1+\varepsilon)^{-1} = \beta$, Theorem \ref{thm:nagaev_weak_array_mob_eps2} applies and, for all $y \geqslant 0$,
\[
\lim_{n \to \infty} \frac{1}{N_n^{2\alpha-1}} \log \Prob(\rva{T}{n} \geqslant N_n^{\alpha} y) = - \frac{y^2}{2\sigma^2} .
\]

\paragraph{Transition $\mathbf{T_0}$}
As in Section \ref{sec:main} after the statement of Theorem \ref{thm:nagaev_weak_array_mob_eps_intermediate}, we define $\theta(y)$ and $y_1$ for the function $f(\theta)=q\theta^{1-\varepsilon} y^{1-\varepsilon}+{(1-\theta)^2y^2/}{(2\sigma^2)}$. Define $\tilde{\theta}(y) \defeq \indic_{y \geqslant y_1} \theta(y)$ and notice that $\tilde{\theta}$ is increasing on $\intervallefo{y_1}{\infty}$ (and $\tilde{\theta}(y) \to 1$ as $y \to \infty$). Set $c_0 \defeq \tilde{\theta}(y_1)y_1 = (2\varepsilon q \sigma^2)^{1/(1+\varepsilon)}$.

\hspace{0.5cm} \textbullet{} When $\alpha = (1+\varepsilon)^{-1} = \beta$ and $c \leqslant c_0$, then
\[
\lim_{n \to \infty} \frac{1}{N_n^{(1-\varepsilon)/(1+\varepsilon)}} \log \Prob(\rva{T}{n} \geqslant N_n^{\alpha} y) = - q k_{0,1}(y) c^{1-\varepsilon} + \frac{(y - k_{0,1}(c)c)^2}{2 \sigma^2} \eqdef -I_{0,1}(y)
\]
where
\[
k_{0,1}(y) \defeq \max\left( \floor{\frac{y-y_{0,1}(c)}{c}} + 1, 0 \right)
\quad \text{and} \quad
y_{0,1}(c) \defeq \frac{c}{2} + q \sigma^2 c^{-\varepsilon}
\]
($y_{0,1}(c)$ is the unique solution in $y$ of $y_{0,1}^2-(y_{0,1}-c)^2 = 2\sigma^2qc^{1-\varepsilon}$).

\hspace{0.5cm} \textbullet{} When $\alpha = (1+\varepsilon)^{-1} = \beta$ and $c \geqslant c_0$, then
\[
\lim_{n \to \infty} \frac{1}{N_n^{(1-\varepsilon)/(1+\varepsilon)}} \log \Prob(\rva{T}{n} \geqslant N_n^{\alpha} y) = - q k_{0,2}(y) c^{1-\varepsilon} + I(y - k_{0,2}(c)c) \eqdef  -I_{0,2}(y)
\]
where
\[
k_{0,2}(y) \defeq \max\left( \floor{\frac{y-y_{0,2}(c)}{c}} + 1, 0 \right)
\quad \text{and} \quad
y_{0,2}(c) \defeq c + (1-\varepsilon)q\sigma^2c^{-\varepsilon}
\]
($y_{0,2}(c)$ is the unique solution in $y$ of $\tilde{\theta}(y)y = c$).

\medskip

Remark: For all $c < c_0$, $y_{0,1}(c) > y_1$: so the Gaussian range in the nontruncated case (which stops at $y_1$) is extended. Moreover, $y_{0,1}(c_0) = y_1 = y_{0,2}(c_0)$, and, for $c = c_0$, $I_{0,1} = I_{0,2}$ (since $I_1(y) = y^2/(2\sigma^2)$ for $y \leqslant y_1$).

\paragraph{Truncated maximal jump range}
When $(1+\varepsilon)^{-1} = \beta < \alpha < \beta+1$ and $y \geqslant 0$, or $\alpha = 1+\beta$ and $y < c$, as before, the proof of Theorem \ref{thm:nagaev_weak_array_mob_eps} adapts and
\[
\lim_{n \to \infty} \frac{1}{N_n^{\alpha-\beta\varepsilon}} \log \Prob(\rva{T}{n} \geqslant N_n^{\alpha} y) = - q y c^{-\varepsilon} .
\]

\paragraph{Trivial case}
When $\alpha=\beta+1$ and $y \geqslant c$, or $\alpha > \beta+1$, 
we obviously have $\Prob(\rva{T}{n} \geqslant N_n^\alpha y) = 0$.

\begin{center}
\begin{figure}
\includegraphics[scale=0.29]{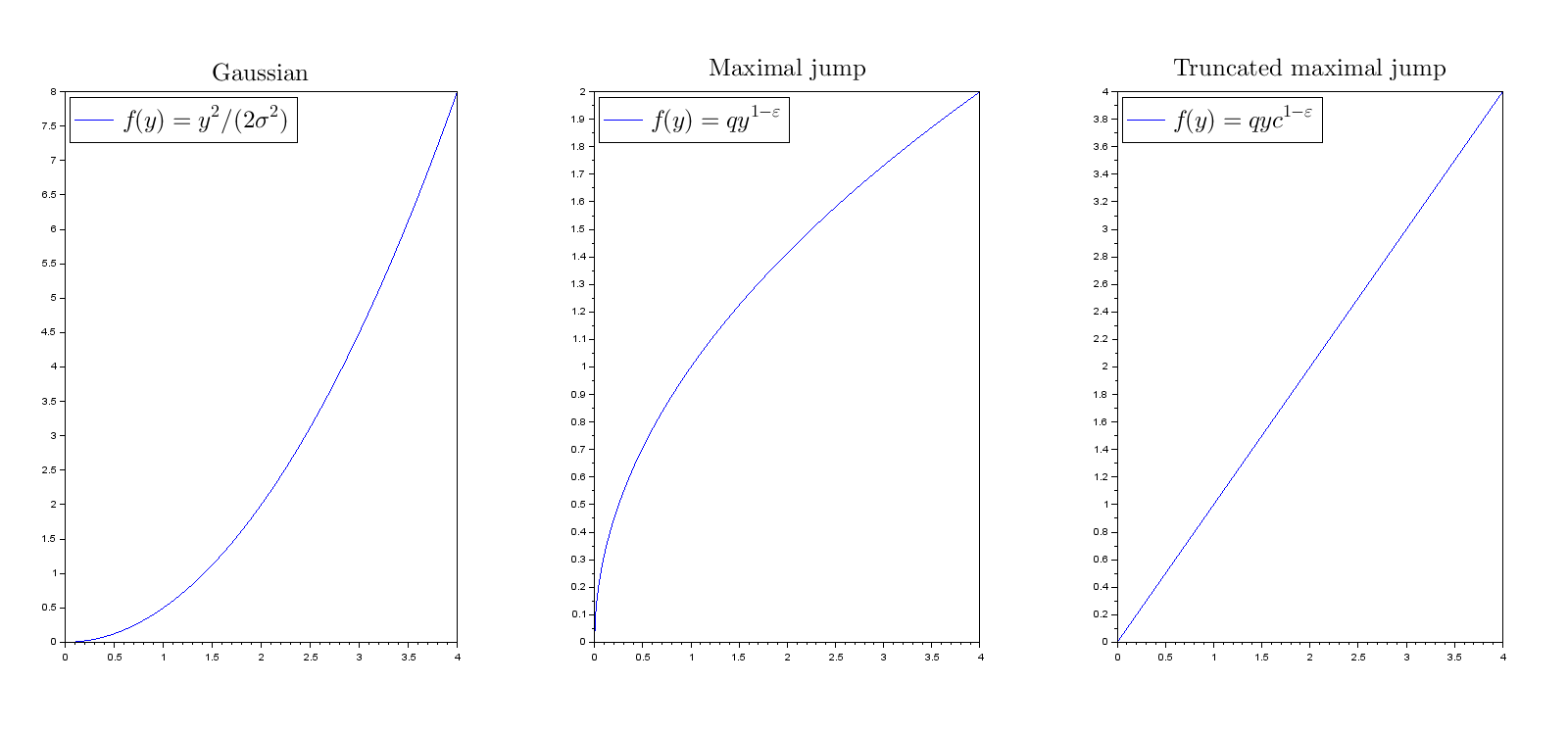}
\caption{\label{fig:fonction_taux_1} Representation of the rate functions. Here, $q=1$, $\sigma^2=2$, $\epsilon =1/2$, and $c=1$.  
Left - Gaussian range. The typical event corresponds to the case where all the random variables are small but their sum has a Gaussian contribution.
Center - Maximal jump range. The typical event corresponds to the case where one random variable contributes to the total sum ($N_n^\alpha y$), no matter the others. We recover the random variable tail.
Right - Truncated maximal jump range. The typical event corresponds to the case where $N_n^{\alpha-\beta}y/c$ variables take the saturation value $N_n^\beta c$, no matter the others.}
\end{figure}
\end{center}

\begin{center}
\begin{figure}
\includegraphics[scale=0.29]{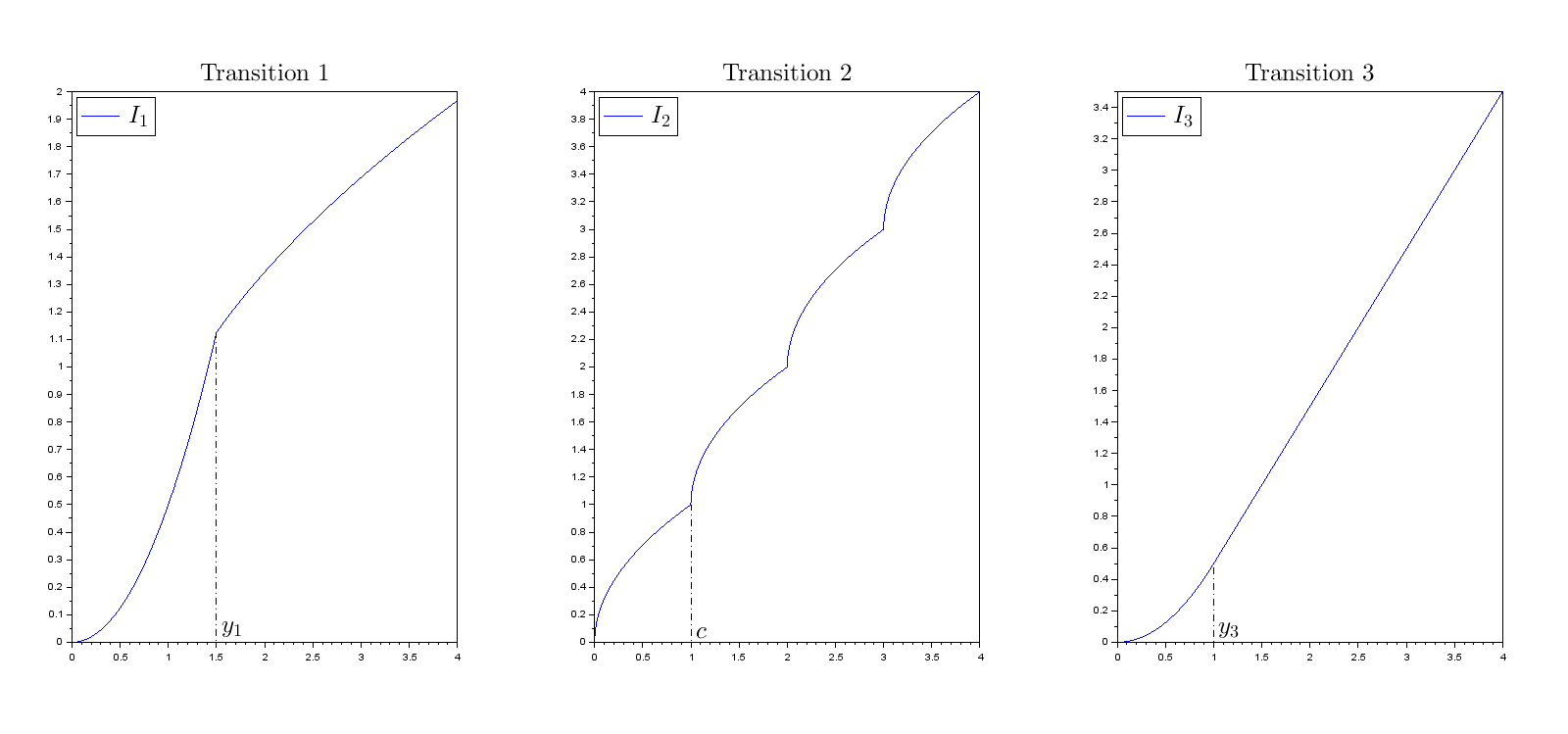}
\caption{\label{fig:fonction_taux_2} Representation of the rate functions. Here, $q=1$, $\sigma^2=2$, $\epsilon =1/2$, and $c=1$.  
Left - Transition 1. The typical event corresponds to the case where one random variable is large ($N_n^\alpha \theta(y) y$) and the sum of the others has a Gaussian contribution (two competing terms).
Center - Transition 2. The typical event corresponds to the case where $\floor{y/c}$ random variables take the saturation value $N_n^\beta c$ and one completes to get the total sum.
Right - Transition 3. The typical event corresponds to the case where some random variables (a number of order $N_n^{1-\beta(1+\varepsilon)}$) take the saturation value $N_n^\beta c$, and the sum of the others has a Gaussian contribution (two competing terms).}
\end{figure}
\end{center}

\begin{center}
\begin{figure}
\includegraphics[scale=0.29]{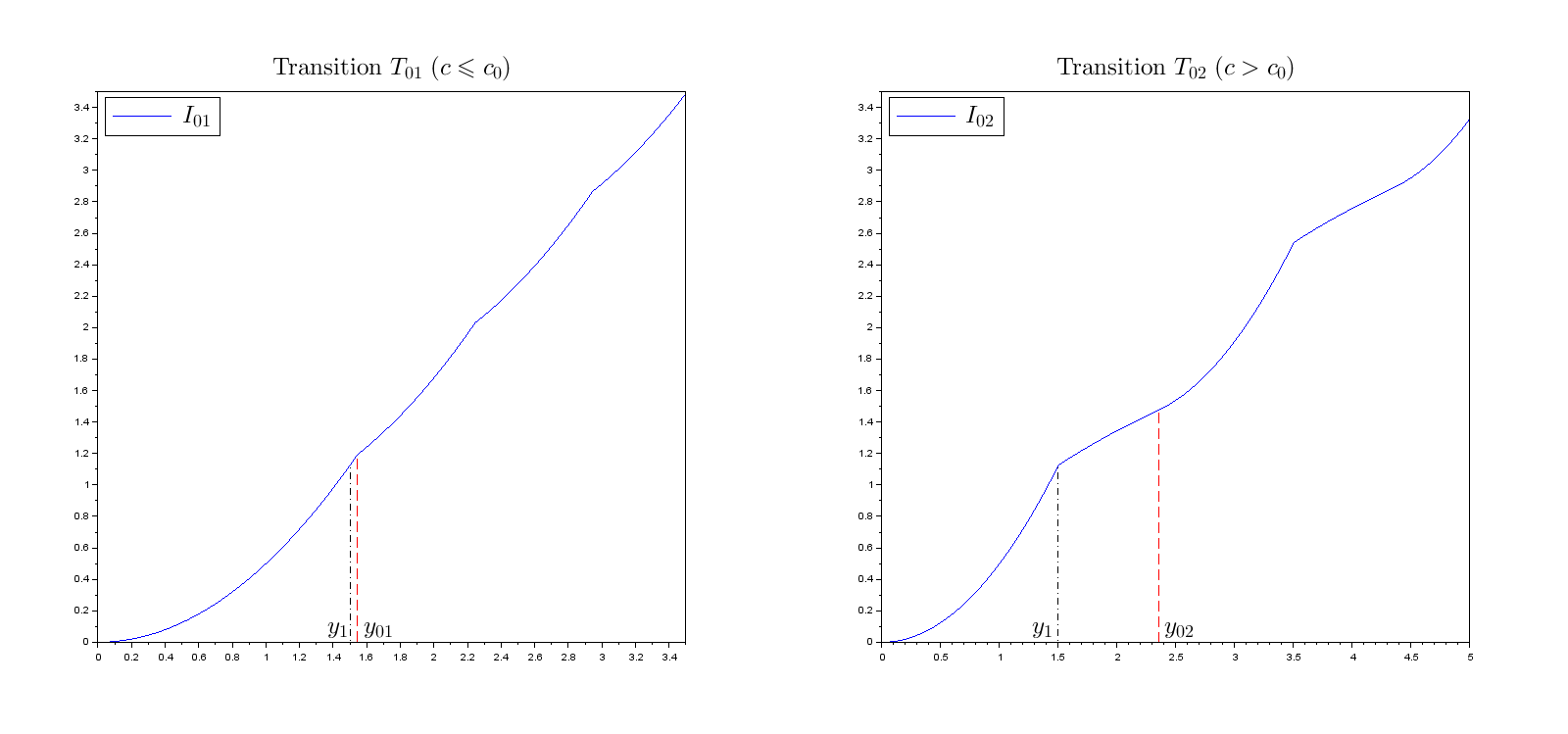}
\caption{\label{fig:fonction_taux_3} Representation of the rate functions. Here, $q=1$, $\sigma^2=2$, $\epsilon =1/2$, and $c=1$.  
Left - Transition 1- for $c \leqslant c_0$. The typical event corresponds to the case where $k_3(c)$ variables take the saturation value $N^\beta c$, and the sum of the others has a Gaussian contribution.
Right - Transition 1- for $c \geqslant c_0$. The typical event corresponds to the case where $k_2(c)$ variables take the saturation value $N^\beta c$, one is also large ($N_n^\beta \theta(y-k_2(c)c) (y-k_2(c)c)$) and the sum of the others has a Gaussian contribution.}
\end{figure}
\end{center}

\begin{center}
\begin{figure}
\begin{tikzpicture}[scale=1.5]
\fill[color=gray](0,2)--(9,2)--(9,8/3)--(8/3,8/3)--(0,4)--cycle;
\fill[color=gray!50](9,8/3)--(9,9)--(8/3,8/3)--cycle;
\fill[color=gray!0](9,9)--(8/3,8/3)--(0,4)--(2,9)--cycle;

\pattern[pattern=north west hatch, hatch distance=3mm, hatch thickness=.5pt] (0,4)--(0,9)--(5,9)--cycle;

\node[draw,fill,color=white,text width=1.2cm,text height=0.7cm,rotate=45]at(1.4,7.2){};

\draw[line width=2pt] (8/3,8/3)--(9,9);
\draw[line width=2pt] (0,4)--(8/3,8/3)--(9,8/3);
\draw[line width=2pt] (0,4)--(5,9);
\draw[line width=1pt,dashed] (0,8/3)--(8/3,8/3);
\draw[line width=1pt,dashed] (8/3,2)--(8/3,8/3);
\draw[line width=2pt,->](0,2)--(9.2,2);
\draw[line width=2pt,->](0,2)--(0.0,9.2);
\node(a)at(9,1.5){$\beta$};
\node(aa)at(0,1.7){$0$};

\node(b)at(-0.5,9){$\alpha$};
\node(bb)at(-0.5,2){\footnotesize{$1/2$}};

\node(d)at(8/3+0.2,1.5){\footnotesize{$(1+\varepsilon)^{-1}$}};
\node(e)at(-0.8,8/3){\footnotesize{$(1+\varepsilon)^{-1}$}};

\node(f)at(-0.5,4){$1$};

\node[rotate=45]at(8.5,8.2){\footnotesize{$\alpha=\beta$}};
\node[rotate=45]at(4.5,8.2){\footnotesize{$\alpha=\beta+1$}};
\node at(8.2,2.5){\footnotesize{$\alpha=(1+\varepsilon)^{-1}$}};
\node[rotate=-27]at(1.2,3.2){\footnotesize{$\alpha=1-\beta\varepsilon$}};

\node(g)at(1.4,2.30){\footnotesize{Gaussian}};
\node(gg)at(6,2.30){\footnotesize{$y^2/(2\sigma^2)$}};
\node(h)[rotate=45]at(1.6,7.1){\footnotesize{$\infty$}};
\node(hh)[rotate=45]at(1.3,7.3){\footnotesize{Trivial}};
\node(i)[rotate=45]at(6.8,4.5){\footnotesize{$qy^{1-\epsilon}$}};
\node(ii)[rotate=45]at(6.6,4.8){\footnotesize{Maximal Jump}};
\node(f)[rotate=45]at(4.8,6.5){\footnotesize{$qyc^{-\epsilon}$}};
\node(ff)[rotate=45]at(4.5,6.6){\footnotesize{Truncated Maximal Jump}};
\draw[line width=1pt,->](8/3,3.4)--(8/3,2.8);
\node(f)at(8/3,3.7){$T_{0}$};
\node(g)at(5.2,8/3+0.2){\footnotesize{Transition 1 \hspace{1cm} $I_1(y)$}};
\node(h)[rotate=45]at(5.5,5.8){\footnotesize{Transition 2 \hspace{0.5cm} $I_2(y)$}};
\node(j)[rotate=-27]at(1.4,3.5){\footnotesize{Transition 3} \hspace{0.2cm} $I_3(y)$};
\end{tikzpicture}
\caption{\label{graph:alpha_beta} Rate function transition diagram.}
\end{figure}
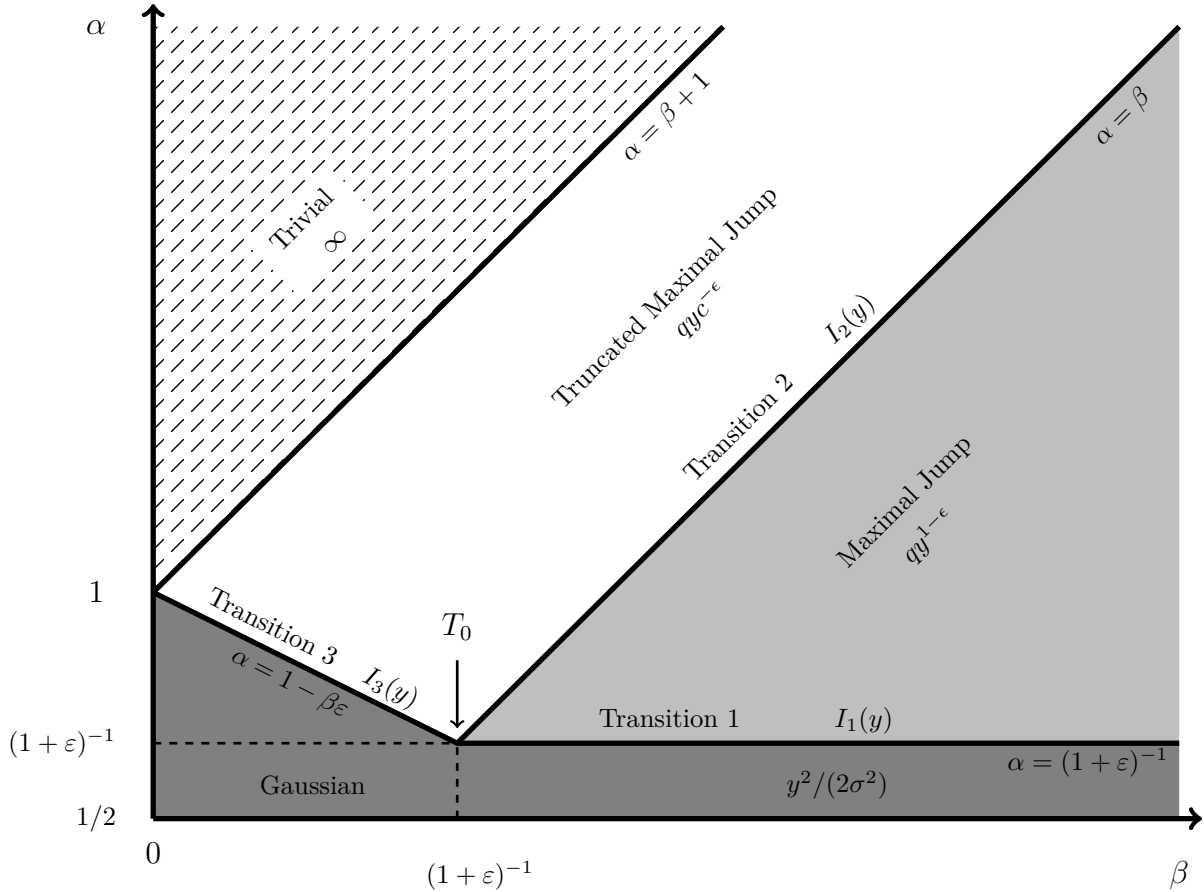
\end{center}

\newpage

\bibliographystyle{abbrv}
\bibliography{biblio_gde_dev}

\end{document}

%% file: packmin.tex
	\usepackage{babel}					

	\usepackage{amsmath}					
	\usepackage{amssymb}					
	\usepackage{amsthm}					

	\usepackage{graphicx}				
	\usepackage{color}					

	\usepackage{version}					

	\usepackage{geometry}				